\documentclass{article}

\usepackage{geometry}
\usepackage{hyperref}
\usepackage{graphicx}
\usepackage{amsthm,amsfonts,amssymb,amsmath}
\usepackage[english]{babel}

\newtheorem{theo}{Theorem}[section]

\newtheorem{lem}[theo]{Lemma}
\newtheorem{conj}[theo]{Conjecture}
\newtheorem{fact}[theo]{Fact}
\newtheorem{claim}[theo]{Claim}
\newtheorem*{claim-no}{Claim}
\newtheorem{coro}[theo]{Corollary}

\theoremstyle{definition}
\newtheorem{defi}[theo]{Definition}
\newtheorem{remark}[theo]{Remark}
\newtheorem{proc}[theo]{Procedure}

\numberwithin{equation}{section}
\numberwithin{figure}{section}

\usepackage[T1]{fontenc}
\usepackage[utf8]{inputenc}

\newcommand{\eps}{\varepsilon}
\newcommand{\e}{\overline{e}}
\newcommand{\C}{\mathcal{C}}
\newcommand{\su}{\subseteq}
\newcommand{\mi}{\setminus}
\newcommand{\tH}{\widetilde{H}}
\newcommand{\es}{\emptyset}
\newcommand{\R}{\text{red}}
\newcommand{\Z}{\mathcal{Z}}
\newcommand{\sh}{\operatorname{sh}}
\renewcommand{\l}{\ell}

\title{A proof of a conjecture of Erd\H{o}s, Faudree, Rousseau and Schelp on subgraphs of minimum degree $k$}
\author{Lisa Sauermann\thanks{Department of Mathematics, Stanford University, 450 Serra Mall, Building 380, Stanford, CA 94305, USA. Email: \texttt{lsauerma@stanford.edu}. Research supported by Jacob Fox's Packard Fellowship.}}

\begin{document}

\maketitle

\begin{abstract}
\noindent Erd\H{o}s, Faudree, Rousseau and Schelp observed the following fact for every fixed integer $k\geq 2$: Every graph on $n\geq k-1$ vertices with at least $(k-1)(n-k+2)+{k-2\choose 2}$ edges contains a subgraph with minimum degree at least $k$. However, there are examples in which the whole graph is the only such subgraph. Erd\H{o}s et al.\ conjectured that having just one more edge implies the existence of a subgraph on at most $(1-\eps_k)n$ vertices with minimum degree at least $k$, where $\eps_k>0$ depends only on $k$. We prove this conjecture, using and extending ideas of Mousset, Noever and \v{S}kori\'{c}.
\end{abstract}

\section{Introduction}

Let $k\geq 2$ be fixed. Erd\H{o}s, Faudree, Rousseau and Schelp \cite{erdos2} observed the following fact.

\begin{fact}\label{fact1} Every graph on $n\geq k-1$ vertices with at least
$$(k-1)(n-k+2)+{k-2\choose 2}$$
edges contains a subgraph of minimum degree at least $k$.
\end{fact}

This fact can be proved very easily by induction: For $n=k-1$ it is vacuously true, because there are no graphs on $k-1$ vertices with the given number of edges. For any $n\geq k$ given a graph with at least $(k-1)(n-k+2)+{k-2\choose 2}$ edges that does not have minimum degree at least $k$, we can delete a vertex of degree at most $k-1$. Then we obtain a graph with $n-1$ vertices and at least $(k-1)((n-1)-k+2)+{k-2\choose 2}$ edges and we can apply the induction assumption.

Erd\H{o}s, Faudree, Rousseau and Schelp \cite{erdos2} also observed that the bound given in Fact \ref{fact1} is sharp and that for each $n\geq k+1$ there exist graphs on $n$ vertices with $(k-1)(n-k+2)+{k-2\choose 2}$ edges, that do not have any subgraphs of minimum degree at least $k$ on fewer than $n$ vertices (an example for such a graph is the generalized wheel formed by a copy of $K_{k-2}$ and a copy of $C_{n-k+2}$ with all edges in between). However, they conjectured in \cite{erdos2} that having just one additional edge implies the existence of a significantly smaller subgraph of minimum degree at least $k$:

\begin{conj}\label{conject} For every $k\geq 2$ there exists $\eps_k>0$ such that each graph on $n\geq k-1$ vertices with
$$(k-1)(n-k+2)+{k-2\choose 2}+1$$
edges contains a subgraph on at most $(1-\eps_k)n$ vertices and with minimum degree at least $k$.
\end{conj}

According to \cite{erdos2}, originally this was a conjecture of Erd\H{o}s for $k=3$. He also included the conjecture for $k=3$ in a list of his favourite problems in graph theory \cite[p.~13]{erdos1}.

Erd\H{o}s, Faudree, Rousseau and Schelp \cite{erdos2} made progress towards Conjecture \ref{conject} and proved that there is a subgraph with minimum degree at least $k$ with at most $n-\lfloor\sqrt{n}/\sqrt{6k^{3}}\rfloor$ vertices. Mousset, Noever and \v{S}kori\'{c} \cite{zurich} improved this to $n-n/(8(k+1)^5\log_2 n)$.

The goal of this paper is to prove Conjecture \ref{conject}. More precisely, we prove the following theorem.

\begin{theo}\label{thm} Let $k\geq 2$ and let $1\leq t\leq \frac{(k-2)(k+1)}{2}-1$ be an integer. Then every graph on $n\geq k-1$ vertices with at least $(k-1)n-t$ edges contains a subgraph on at most
$$\left(1-\frac{1}{\max(10^{4}k^2, 100kt)}\right)n$$
vertices and with minimum degree at least $k$.
\end{theo}

For $t=\frac{(k-2)(k+1)}{2}-1$ we have
$$(k-1)n-t=(k-1)n-\frac{(k-2)(k+1)}{2}+1=(k-1)(n-k+2)+{k-2\choose 2}+1$$
and so Theorem \ref{thm} implies Conjecture \ref{conject} with
$$\eps_k=\frac{1}{\max\left(10^{4}k^2, 100k\left(\frac{(k-2)(k+1)}{2}-1\right)\right)}>\frac{1}{10^{4}k^3}.$$

On the other hand, for example for $t=1$ we obtain the following statement: Every graph on $n\geq k-1$ vertices with at least $(k-1)n-1$ edges contains a subgraph on at most
$$\left(1-\frac{1}{10^{4}k^2}\right)n$$
vertices and with minimum degree at least $k$. Thus, the presence of one additional edge compared to the number in Fact \ref{fact1} implies the existence of a subgraph with minimum degree at least $k$ on $(1-\eps)n$ vertices with $\eps=\Omega(k^{-3})$, while the presence of $\frac{(k-2)(k+1)}{2}$ additional edges (which is a fixed number with respect to $n$) already gives $\eps=\Omega(k^{-2})$.

The basic approach to proving Theorem \ref{thm} is to assign colours to some vertices, such that for every colour the subgraph remaining after deleting all vertices of that colour has minimum degree at least $k$. If we can ensure that sufficiently many vertices get coloured (and the number of colours is fixed), then in this way we find a significantly smaller subgraph with minimum degree at least $k$.

Our proof relies on and extends the ideas of Mousset, Noever and \v{S}kori\'{c} in \cite{zurich}, although they do not use a colouring approach in their argument. We construct our desired colouring iteratively, and apply the techniques of Mousset, Noever and \v{S}kori\'{c} in every step of the iteration.

Note that if $G$ has a subgraph with minimum degree at least $k$, then the induced subgraph with the same vertices also has minimum degree at least $k$. Thus, in all the statements above we can replace `subgraph' by `induced subgraph'.

\textit{Organization.} We prove Theorem \ref{thm} in Section \ref{sect2} apart from the proof of a certain lemma. This lemma is proved in Section \ref{sect3} assuming another lemma, Lemma \ref{mainlem}, that is stated at the beginning of Section 3. Lemma \ref{mainlem} is a key tool in our proof and an extension of Lemma 2.7 in Mousset, Noever and \v{S}kori\'{c}'s paper \cite{zurich}. We prove Lemma \ref{mainlem} in Section \ref{sect5} after some preparations in Section \ref{sect4}.

\textit{Notation.} All graphs throughout this paper are assumed to be finite. Furthermore, all subgraphs are meant to be non-empty (but they may be equal to the original graph). An induced subgraph is called proper if it has fewer vertices than the original graph. When we say that a graph has minimum degree at least $k$, we implicitly also mean to say that the graph is non-empty.

For a graph $G$ let $V(G)$ denote the set of its vertices, $v(G)$ the number of its vertices and $e(G)$ the number of its edges. For any integer $i$, let $V_{\leq i}(G)$ be the set of vertices of $G$ with degree at most $i$ and $V_{i}(G)$ the set of vertices of $G$ with degree equal to $i$.

For a subset $X\su V(G)$ let $G-X$ be the graph obtained by deleting all vertices in $X$ from $G$. Let $\e_G(X)$ denote the number of edges of $G$ that are incident with at least one vertex in $X$, i.e.\ $\e_G(X)=e(G)-e(G-X)$. Call a vertex $v\in V(G)$ adjacent to $X$, if $v$ is adjacent to at least one vertex in $X$. In this case, call $X$ a neighbour of $v$ and $v$ a neighbour of $X$. Finally, if $\mathcal{X}$ is a collection of disjoint subsets $X\su V(G)$, then $G-\mathcal{X}$ denotes the graph obtained from $G$ by deleting all members of $\mathcal{X}$.

In general, we try to keep the notation and the choice of variables similar to \cite{zurich}, so that the reader can see the connections to the ideas in \cite{zurich} more easily.

\section{Proof of Theorem \ref{thm}} \label{sect2}

Let $k\geq 2$ be fixed throughout the paper. Furthermore, let $1\leq t\leq \frac{(k-2)(k+1)}{2}-1$ be an integer and let 
$$\eps=\frac{1}{\max(10^{4}k^2, 100kt)}.$$

We prove Theorem \ref{thm} by induction on $n$. Note that the theorem is vacuously true for $n=k-1$, because in this case a graph on $n$ vertices can have at most
\[{k-1\choose 2}=(k-1)^2-\frac{(k-1)k}{2}=(k-1)n-\frac{(k-2)(k+1)}{2}-1<(k-1)n-t\]
edges. So from now on, let us assume that $n\geq k$ and that Theorem \ref{thm} holds for all smaller values of $n$.

Consider a graph $G$ on $n$ vertices with $e(G)\geq (k-1)n-t$ edges. We would like to show that $G$ contains a subgraph on at most $(1-\eps)n$ vertices and with minimum degree at least $k$. So let us assume for contradiction that $G$ does not contain such a subgraph.

\begin{claim}\label{claim-sets-many-edges}For every subset $X\su V(G)$ of size $1\leq \vert X\vert\leq n-k+1$, we have $\e_G(X)\geq (k-1)\vert X\vert+1$.
\end{claim}

\begin{proof}Suppose there exists a subset $X\su V(G)$ of size $1\leq \vert X\vert\leq n-k+1$ with $\e_G(X)\leq (k-1)\vert X\vert$. Then let $G'=G-X$ be the graph obtained by deleting $X$. Note that $G'$ has $n-\vert X\vert\geq k-1$ vertices and 
\[e(G')=e(G)-\e_G(X)\geq (k-1)n-t-(k-1)\vert X\vert=(k-1)(n-\vert X\vert)-t.\]
So by the induction assumption $G'$ contains a subgraph on at most $(1-\eps)(n-\vert X\vert)\leq (1-\eps)n$ vertices and with minimum degree at least $k$. This subgraph is also a subgraph of $G$, which is a contradiction to our assumption on $G$.
\end{proof}

In particular, for every vertex $v\in V(G)$, Claim \ref{claim-sets-many-edges} applied to $X=\lbrace v\rbrace$ yields $\deg(v)=\e_G(\lbrace v\rbrace)\geq k$. Thus, the graph $G$ has minimum degree at least $k$ and $n\geq k+1$.

The following lemma is very similar to Lemma 4 in \cite{erdos2} for $\alpha=\frac{1}{3k}$. However, in contrast to \cite[Lemma 4]{erdos2}, in our situation we are not given an upper bound on the number of edges.

\begin{lem}[see Lemma 4 in \cite{erdos2}]\label{deg-k-vertices} Let $H$ be a graph on $n$ vertices with minimum degree at least $k$. If $H$ has at most $\frac{n}{3k}$ vertices of degree $k$, then $H$ has a subgraph on at most $\left(1-\frac{1}{27k^{2}}\right)n$ vertices with minimum degree at least $k$.
\end{lem}

One can prove this Lemma in a similar way as \cite[Lemma 4]{erdos2}. For the reader's convenience, we provide a proof in the appendix.

Recall our assumption that $G$ does not have a subgraph on at most $(1-\eps)n$ vertices with minimum degree at least $k$, and also recall that $G$ has minimum degree at least $k$. By Lemma \ref{deg-k-vertices} the graph $G$ must therefore have at least $\frac{n}{3k}$ vertices of degree $k$. Furthermore, as $\eps<\frac{1}{2}$, the graph $G$ must be connected.

A very important idea for the proof is the notion of \emph{good sets} from \cite{zurich}. Both the following definition and the subsequent properties are taken from \cite{zurich}. We repeat the proofs here, because our statements differ slightly from the ones in \cite{zurich}, but the proofs are essentially the same.

\begin{defi}[Definition 2.1 in \cite{zurich}]\label{def-good-set} A \emph{good set} in $G$ is a subset of $V(G)$ constructed according to the following rules:
\begin{itemize}
\item If $v$ is a vertex of degree $k$ in $G$, then $\lbrace v\rbrace$ is a good set.
\item If $A$ is a good set and $v\in V(G)\mi A$ with $\deg_{G-A}(v)\leq k-1$, then $A\cup \lbrace v\rbrace$ is a good set.
\item If $A$ and $B$ are good sets and $A\cap B\neq \es$, then $A\cup B$ is a good set.
\item If $A$ and $B$ are good sets and there is an edge connecting a vertex in $A$ to a vertex in $B$, then $A\cup B$ is a good set.
\end{itemize}
\end{defi}

Clearly, each good set is non-empty.

\begin{lem}[see Claim 2.2(i) in \cite{zurich}]\label{lemma-edges-good-set} If $D$ is a good set in $G$ of size $\vert D\vert\leq n-k+1$, then $\e_G(D)\leq (k-1)\vert D\vert+1$.
\end{lem}
\begin{proof}We prove the lemma by induction on the construction rules in Definition \ref{def-good-set}. For the first rule, observe that $\e_G(\lbrace v\rbrace)=\deg_G(v)= (k-1)+1$ for every vertex $v$ of degree $k$. For the second rule, assume that $\e_G(A)\leq (k-1)\vert A\vert+1$ and $\deg_{G-A}(v)\leq k-1$, then
$$\e_G(A\cup \lbrace v\rbrace)=\e_G(A)+\deg_{G-A}(v)\leq (k-1)(\vert A\vert+1)+1.$$
For the third rule, assume that $\e_G(A)\leq (k-1)\vert A\vert+1$ and $\e_G(B)\leq (k-1)\vert B\vert+1$ as well as $\vert A\vert, \vert B\vert\leq n-k+1$ and $A\cap B\neq \emptyset$. Then $1\leq \vert A\cap B\vert\leq n-k+1$ and therefore by Claim \ref{claim-sets-many-edges} we have $\e_G(A\cap B)\geq (k-1)\vert A\cap B\vert+1$. Thus,
$$\e_G(A\cup B)\leq \e_G(A)+\e_G(B)-\e_G(A\cap B)\leq (k-1)(\vert A\vert+\vert B\vert-\vert A\cap B\vert)+1=(k-1)\vert A\cup B\vert+1.$$
For the fourth rule, assume that $\e_G(A)\leq (k-1)\vert A\vert+1$ and $\e_G(B)\leq (k-1)\vert B\vert+1$ and that there is at least one edge between $A$ and $B$. We may also assume that $A$ and $B$ are disjoint, because the case $A\cap B\neq \emptyset$ is already covered by the third rule. Now
$$\e_G(A\cup B)\leq \e_G(A)+\e_G(B)-1\leq (k-1)(\vert A\vert+\vert B\vert)+1=(k-1)\vert A\cup B\vert+1.$$
This finishes the proof of the lemma.\end{proof}

\begin{claim}[see Claim 2.2(ii) in \cite{zurich}]\label{remove-good-set-1} If $D$ is a good set in $G$ of size $\vert D\vert\leq n-k+1$, then $G-D$ contains a subgraph of minimum degree at least $k$.
\end{claim}
\begin{proof} The graph $G-D$ has $n-\vert D\vert\geq k-1$ vertices and $e(G)-\e_G(D)$ edges. By Lemma \ref{lemma-edges-good-set} these are at least
\begin{multline*}
e(G)-\e_G(D)\geq ((k-1)n-t)-((k-1)\vert D\vert+1)=(k-1)(n-\vert D\vert)-(t+1)\\
\geq (k-1)(n-\vert D\vert)-\frac{(k-2)(k+1)}{2}=(k-1)((n-\vert D\vert)-k+2)+{k-2\choose 2}
\end{multline*}
edges, so by Fact \ref{fact1} the graph $G-D$ contains a subgraph of minimum degree at least $k$.
\end{proof}

\begin{claim}[see the paragraph of \cite{zurich} below the proof of Claim 2.2]\label{no-big-good-set} If $D$ is a good set in $G$, then $\vert D\vert\leq \frac{n}{k}$.\end{claim}
\begin{proof} Suppose there is a good set $D$ of size $\vert D\vert> \frac{n}{k}>1$ (recall $n\geq k+1$). Then let us choose such a set with $\vert D\vert$ minimal (under the constraint $\vert D\vert> \frac{n}{k}$). As $D$ is constructed according to the rules in Definition \ref{def-good-set}, there must be a good set $D'$ with size $\vert D'\vert=\vert D\vert-1$ or $\vert D\vert/2\leq \vert D'\vert<\vert D\vert$. As $\vert D\vert\geq 2$ we have in either case $\vert D'\vert\geq \vert D\vert/2\geq \frac{n}{2k}$, but also $\vert D'\vert<\vert D\vert$ and therefore $\vert D'\vert\leq \frac{n}{k}$ by the choice of $D$. As $n\geq k+1$, this implies $\vert D'\vert\leq \frac{n}{k}\leq n-k+1$ and so by Claim \ref{remove-good-set-1} the graph $G-D'$ contains a subgraph with minimum degree at least $k$. But this subgraph has at most $n-\vert D'\vert\leq n-\frac{n}{2k}<(1-\eps)n$ vertices, which is a contradiction to our assumption on $G$.
\end{proof}

\begin{coro}\label{coro-edges-good-set}For every good set $D$ in $G$ we have $\e_G(D)\leq (k-1)\vert D\vert+1$.
\end{coro}
\begin{proof}As $\frac{n}{k}\leq n-k+1$, this follows directly from Claim \ref{no-big-good-set} and Lemma \ref{lemma-edges-good-set}.
\end{proof}

Now denote by $\C^{*}$ the collection of all maximal good subsets of $G$. By Corollary \ref{coro-edges-good-set} every $D\in \C^{*}$ is a subset of $V(G)$ satisfying $\e_G(D)\leq (k-1)\vert D\vert+1$. Furthermore, by the construction rules in Definition \ref{def-good-set} all the members of $\C^{*}$ are disjoint and there are no edges between them.

\begin{claim}\label{remove-good-set} If $D\in \C^{*}$, then $G-D$ is a non-empty graph with minimum degree at least $k$.\end{claim}
\begin{proof} By Claim \ref{no-big-good-set} we have $\vert D\vert\leq \frac{n}{k}<n$, hence $G-D$ is non-empty. Since $D$ is a maximal good set, we have $\deg_{G-D}(v)\geq k$ for each $v\in V(G)\mi D$ by the second rule in Definition \ref{def-good-set}.
\end{proof}

By the first rule in Definition \ref{def-good-set} every vertex of $G$ with degree $k$ is contained in some member of $\C^{*}$. We know by Lemma \ref{deg-k-vertices} that $G$ has at least $\frac{n}{3k}$ vertices of degree $k$. Hence
$$\sum_{D\in \C^{*}}\vert D\vert\geq \frac{n}{3k}.$$

For any subset $\mathcal{X}$ of $\C^{*}$, let $\Vert \mathcal{X}\Vert=\sum_{D\in \mathcal{X}}\vert D\vert$ be the sum of the sizes of the elements of $\mathcal{X}$. Then the last inequality reads $\Vert\C^{*}\Vert=\sum_{D\in \C^{*}}\vert D\vert\geq \frac{n}{3k}$.

Let $m=\vert \C^{*}\vert$ and $\C^{*}=\lbrace D_1,D_2,\dots,D_m\rbrace$ with $\vert D_1 \vert\geq \vert D_2 \vert\geq \dots \geq \vert D_m \vert>0$.

Now let $J$ be the largest positive integer with $2^{J}-1\leq m$. For each $j=1,\dots,J$ set
$$\C_j=\lbrace D_i \,\vert\, 2^{j-1}\leq i<2^{j}\rbrace$$
and set
$$\C=\lbrace D_i \,\vert\, 1\leq i\leq 2^{J}-1\rbrace.$$
Then $\C_1$,\dots, $\C_J$ are disjoint subcollections of $\C^{*}$ and their union is $\C$. For each $j=1,\dots,J$ we have $\vert \C_j\vert=2^{j-1}$.

Note that $\vert \C\vert=2^{J}-1$ and therefore $\vert \C^{*}\mi \C\vert=m-(2^{J}-1)<2^{J}$ by the choice of $J$. Thus, $\vert \C^{*}\mi \C\vert\leq \vert \C\vert$. Since every element of $\C$ has at least the size of every element of $\C^{*}\mi \C$, this implies $\Vert \C^{*}\mi \C\Vert\leq \Vert \C\Vert$ and hence $\Vert \C\Vert\geq \frac{1}{2} \Vert\C^{*}\Vert\geq \frac{n}{6k}$.

Furthermore for each $j=1,\dots,J-1$ we have $\vert \C_{j+1}\vert=2^{j}=2\vert \C_{j}\vert$ and since every element of $\C_j$ has at least the size of every element of $\C_{j+1}$, this implies $\Vert \C_{j+1}\Vert\leq 2\Vert \C_j\Vert$.

Let $J'\leq J$ be the least positive integer such that
\begin{equation}\label{eqJstrich1}
\Vert \C_1\Vert+\dots+\Vert \C_{J'}\Vert\geq \frac{n}{100k}
\end{equation}
(note that $J$ has this property since $\Vert \C_1\Vert+\dots+\Vert \C_J\Vert=\Vert \C\Vert\geq \frac{n}{6k}$, hence there is a least positive integer $J'\leq J$ with the property). Note that $J'>1$, since otherwise we would have $\vert D_1\vert=\Vert \C_1\Vert\geq \frac{n}{100k}$ and since $G-D_1$ has minimum degree at least $k$ by Claim \ref{remove-good-set}, this would mean that $G$ has a subgraph with at most $(1-\frac{1}{100k})n$ vertices and minimum degree at least $k$, which contradicts our assumptions on $G$. Thus, indeed $J'>1$.

Note that $\Vert \C_1\Vert+\dots+\Vert \C_{J'-1}\Vert< \frac{n}{100k}$ by the choice of $J'$ and in particular $\Vert \C_{J'-1}\Vert< \frac{n}{100k}$. Hence
\begin{equation}\label{eqJstrich3}
\Vert \C_{J'}\Vert\leq 2\Vert \C_{J'-1}\Vert< \frac{2n}{100k}
\end{equation}
and therefore
$$\Vert \C_1\Vert+\dots+\Vert \C_{J'}\Vert<\frac{n}{100k}+\frac{2n}{100k}=\frac{3n}{100k}.$$
Thus,
\begin{equation}\label{eqJstrich2}
\Vert \C_{J'+1}\Vert+\dots+\Vert \C_{J}\Vert=\Vert \C\Vert-(\Vert \C_1\Vert+\dots+\Vert \C_{J'}\Vert)\geq \frac{n}{6k}-\frac{3n}{100k}>\frac{n}{8k}.
\end{equation}

\begin{claim}\label{Jstrichgross} $2^{J'}>t$.
\end{claim}
\begin{proof} By (\ref{eqJstrich1}) we have
$$\sum_{D\in \C_1\cup\dots\cup\C_{J'}}\vert D\vert=\Vert \C_1\Vert+\dots+\Vert \C_{J'}\Vert\geq \frac{n}{100k}.$$
Note that the set $\C_1\cup\dots\cup\C_{J'}=\lbrace D_i \,\vert\, 1\leq i<2^{J'}\rbrace$ has $2^{J'}-1$ elements. If we had $2^{J'}\leq t$, then the above sum would have at most $t-1$ summands and in particular there would be some $D\in \C_1\cup\dots\cup\C_{J'}\su \C^{*}$ with $\vert D\vert\geq \frac{n}{100kt}$. But then by Claim \ref{remove-good-set} the graph $G-D$ would be a subgraph of $G$ with minimum degree at least $k$ and at most $(1-\frac{1}{100kt})n\leq (1-\eps)n$ vertices. This is a contradiction to our assumption on $G$.
\end{proof}

Let us now fix $401k$ colours and enumerate them colour 1 to colour $401k$. As indicated in the introduction, we prove Theorem \ref{thm} by constructing a colouring of some of the vertices of $G$, such that when removing the vertices of any colour class what remains is a graph of minimum degree at least $k$. If we colour sufficiently many vertices of $G$, then one of the colour classes has size at least $\eps n$ and we obtain a subgraph on at most $(1-\eps) n$ vertices with minimum degree at least $k$.

Throughout the paper, when we talk about colourings of the vertices of a graph we do not mean proper colourings (i.e.\ we allow adjacent vertices to have the same colour) and we also allow some vertices to remain uncoloured.

\begin{defi}\label{approcolour} For $\l=J',\dots,J$, an \emph{$\l$-appropriate} colouring of $G$ is a colouring of some of the vertices of $G$ by the $401k$ given colours, such that the following seven conditions are satisfied. Here, for each $i=1,\dots,401k$, we let $X_i\su V(G)$ denote the set of vertices coloured in colour $i$.
\begin{itemize}
\item[(i)] Each vertex has at most one colour.
\item[(ii)] For each $D\in \C$, the set $D$ is either monochromatic or completely uncoloured.
\item[(iii)] For $i=1,\dots,401k$, let $y_i^{(\l)}$ be the number of members of $\C_1\cup\dots\cup\C_\l$ that are monochromatic in colour $i$. Then for each $i=1,\dots,401k$ we have $\e_G(X_i)\leq (k-1)\vert X_i\vert+y_i^{(\l)}$.
\item[(iv)] For each $i=1,\dots,401k$, the graph $G-X_i$ has minimum degree at least $k$. In other words: When removing all the vertices with colour $i$, we obtain a graph of minimum degree at least $k$.
\item[(v)] The members of $\C_1\cup\dots\cup\C_{J'}$ are all uncoloured.
\item[(vi)] For every $J'+1\leq j\leq \l$, the number of uncoloured members of $\C_j$ is at most $\frac{1}{4}\vert \C_j\vert$.
\item[(vii)] If $D\in \C_{\l+1}\cup\dots\cup\C_{J}$ is uncoloured, then all of its neighbours $v\in V(G)$ are also uncoloured.
\end{itemize}
\end{defi}

\begin{lem}\label{colour} For every $\l=J',\dots,J$ there exists an $\l$-appropriate colouring of $G$.
\end{lem}

We prove Lemma \ref{colour} by induction in Section \ref{sect3}.

We remark that in order to prove Theorem \ref{thm} we only need Lemma \ref{colour} for $\l=J$ and we also only need conditions (ii), (iv) and (vi) in Definition \ref{approcolour}. However, all the other conditions are needed in order to keep the induction in the proof of Lemma \ref{colour} running.

Let us now finish the proof of Theorem \ref{thm}. Consider a $J$-appropriate colouring of $G$, which exists by Lemma \ref{colour}.

For every $j=J'+1,\dots,J$ by condition (vi) in Definition \ref{approcolour} the number of uncoloured members of $\C_j$ is at most $\frac{1}{4}\vert \C_j\vert=\frac{1}{4}2^{j-1}=\frac{1}{2}2^{j-2}=\frac{1}{2}\vert \C_{j-1}\vert$. Since the size of every member of $\C_{j-1}$ is at least the size of every member of $\C_j$, this implies
$$\sum_{D\in \C_j\text{ uncoloured}}\vert D\vert\leq \frac{1}{2}\Vert \C_{j-1}\Vert.$$
Thus, for every $j=J'+1,\dots,J$,
$$\sum_{D\in \C_j\text{ coloured}}\vert D\vert=\Vert \C_j\Vert-\sum_{D\in \C_j\text{ uncoloured}}\vert D\vert\geq \Vert \C_j\Vert-\frac{1}{2}\Vert \C_{j-1}\Vert.$$
So the total number of coloured vertices is at least
$$\sum_{D\in \C_{J'+1}\cup\dots\cup \C_J\text{ coloured}}\vert D\vert\geq \sum_{j=J'+1}^{J} \left(\Vert \C_j\Vert-\frac{1}{2}\Vert \C_{j-1}\Vert\right)=
\sum_{j=J'+1}^{J} \Vert \C_j\Vert-\frac{1}{2}\sum_{j=J'+1}^{J-1} \Vert \C_j\Vert-\frac{1}{2}\Vert \C_{J'}\Vert,$$
and by (\ref{eqJstrich3}) and (\ref{eqJstrich2})  this is at least
$$\frac{1}{2}\sum_{j=J'+1}^{J} \Vert \C_j\Vert-\frac{1}{2}\cdot \frac{2n}{100}>\frac{n}{16k}-\frac{n}{100k}>\frac{n}{20k}.$$
Since there are at least $\frac{n}{20k}$ coloured vertices, one of the $401k$ colours must occur at least $\frac{n}{10^{4}k^{2}}$ times. If we delete all vertices of this colour, then by condition (iv) in Definition \ref{approcolour} the remaining graph has minimum degree at least $k$ and at most
$$\left(1-\frac{1}{10^{4}k^{2}}\right)n\leq (1-\eps)n$$
vertices. This contradicts our assumption on $G$ and finishes the proof of Theorem \ref{thm}.

\section{Proof of Lemma \ref{colour}} \label{sect3}

The goal of this section is to prove Lemma \ref{colour}. Our proof proceeds by induction on $\l$. In each step we apply the following key lemma, which is an extension of \cite[Lemma 2.7]{zurich}.

\begin{lem}\label{mainlem} Let $H$ be a graph and $\C_H$ be a collection of disjoint non-empty subsets of $V(H)$, such that for each $D\in\C_H$ we have $\e_H(D)\leq (k-1)\vert D\vert+1$ and $\deg_H(v)\geq k$ for each $v\in D$. Assume that $H$ does not have a subgraph of minimum degree at least $k$.

Then we can find a subset $S\su V_{\leq k-1}(H)$ with $ V_{\leq k-2}(H)\su S$ and
$$\sum_{s\in S}(k-\deg_H(s))\leq 2((k-1)v(H)-e(H))$$
as well as disjoint subsets $B_v\su V(H)$ for each vertex $v\in V_{\leq k-1}(H)\mi S=V_{k-1}(H)\mi S$, such that the following holds:

If $\tH$ is a graph containing $H$ as a proper induced subgraph such that $V_{\leq k-1}(\tH)\su V_{\leq k-1}(H)\mi S$ and no vertex in $V(\tH)\mi V(H)$ is adjacent to any member of $\C_H$, then there exists a (non-empty) induced subgraph $\tH'$ of $\tH$ with the following six properties:
\begin{itemize}
\item[(a)] The minimum degree of $\tH'$ is at least $k$.
\item[(b)] $(k-1)v(\tH')-e(\tH')\leq (k-1)v(\tH)-e(\tH)$.
\item[(c)] $V(\tH)\mi V(\tH')\su \bigcup_{v\in V_{\leq k-1}(\tH)}B_v$, so in particular $V(\tH)\mi V(\tH')\su V(H)$.
\item[(d)] No vertex in $V(\tH)\mi V(\tH')$ is adjacent to any vertex in $V(\tH)\mi V(H)$.
\item[(e)] For each $D\in \C_H$ either $D\su V(\tH')$ or $D\cap V(\tH')=\es$.
\item[(f)] If $D\in \C_H$ and $D\su V(\tH')$, then $D$ is not adjacent to any vertex in $V(\tH)\mi V(\tH')$.
\end{itemize}
\end{lem}

We prove Lemma \ref{mainlem} in Section \ref{sect5} after some preparations in Section \ref{sect4}.

Let us now prove Lemma \ref{colour}. First notice that Lemma \ref{colour} is true for $\l=J'$. Indeed, we can take the colouring of $G$ in which all vertices are uncoloured. This satisfies all conditions in Definition \ref{approcolour} (note that condition (vi) is vacuous for $\l=J'$).

Now let $J'\leq \l\leq J-1$ and assume that we are given an $\l$-appropriate colouring of $G$. We would like to extend this colouring by colouring some of the yet uncoloured vertices to obtain an $(\l+1)$-appropriate colouring. This would complete the induction step.

In order to avoid later confusion, let us denote the given $\l$-appropriate colouring of $G$ by $\varphi$. As in Definition \ref{approcolour}, for $i=1,\dots,401k$, let $X_i$ be the set of vertices coloured in colour $i$ and $y_i^{(\l)}$ the number of members of $\C_1\cup\dots\cup\C_\l$ that are monochromatic in colour $i$. By (v) all the coloured members of $\C_1\cup\dots\cup\C_\l$ are actually in $\C_{J'+1}\cup\dots\cup\C_\l$ and hence
\begin{equation}\label{y-sum}
y_1^{(\l)}+y_2^{(\l)}+\dots+y_{401k}^{(\l)}+\vert\lbrace D\in \C_{J'+1}\cup\dots\cup\C_\l\,\vert\, D\text{ uncoloured}\rbrace\vert=\vert\C_{J'+1}\vert+\dots+\vert\C_\l\vert.
\end{equation}

Let $\C_{\l+1}'\su \C_{\l+1}$ consist of those members of $\C_{\l+1}$ that are uncoloured in $\varphi$.

\begin{claim}\label{Clplus1strich} If $\vert\C_{\l+1}'\vert\leq \frac{1}{4}\vert\C_{\l+1}\vert$, then the $\l$-appropriate colouring $\varphi$ is $(\l+1)$-appropriate.
\end{claim}
\begin{proof}Conditions (i), (ii), (iv) and (v) in Definition \ref{approcolour} do not depend on the value of $\l$ and are satisfied for $\varphi$. In condition (iii) we have $y_i^{(\l+1)}\geq y_i^{(\l)}$ for each colour $i$ and therefore
$$\e_G(X_i)\leq (k-1)\vert X_i\vert+y_i^{(\l)}\leq (k-1)\vert X_i\vert+y_i^{(\l+1)}.$$
Note that condition (vi) is already satisfied for $J'+1\leq j\leq \l$ and is satisfied for $j=\l+1$ according to the assumption $\vert\C_{\l+1}'\vert\leq \frac{1}{4}\vert\C_{\l+1}\vert$. Condition (vii) is a strictly weaker statement for $\l+1$ than for $\l$.
\end{proof}

Hence we may assume that $\vert\C_{\l+1}'\vert> \frac{1}{4}\vert\C_{\l+1}\vert$.

Set
\begin{equation}\label{H-def1}
H=G-(\C_1\cup\dots\cup\C_{\l+1})-(X_1\cup\dots\cup X_{401k}),
\end{equation}
i.e.\ $H$ is obtained from $G$ by deleting all members of $\C_1\cup\dots\cup\C_{\l+1}$ and all coloured vertices. Note that this can also be expressed as
\begin{equation}\label{H-def2}
H=G-(X_1\cup\dots\cup X_{401k})-(\C_1\cup\dots\cup\C_{J'})-\C_{\l+1}'-\bigcup_{\substack{D\in \C_{J'+1}\cup\dots\cup\C_{\l}\\ \text{uncoloured in }\varphi}}D
\end{equation}
and here all the deleted sets of vertices are disjoint.

\begin{claim}\label{H-nonempty} The graph $H$ is non-empty.
\end{claim}
\begin{proof} By Claim \ref{Clplus1strich} we may assume $\vert\C_{\l+1}'\vert> \frac{1}{4}\vert\C_{\l+1}\vert$. In particular, $\C_{\l+1}'$ is non-empty, so let $D\in \C_{\l+1}'\su \C_{\l+1}$. Since $G$ is connected, $D$ is adjacent to some vertex $v\in V(G)\mi D$. Since there are no edges between the members of $\C$, the vertex $v$ cannot lie in any member of $\C_1\cup\dots\cup\C_{\l+1}$. Furthermore $D\in \C_{\l+1}'$ is uncoloured in the $\l$-appropriate colouring $\varphi$ and by condition (vii) from Definition \ref{approcolour}, this implies that $v$ is also uncoloured in $\varphi$. Hence $v\not\in X_1\cup\dots\cup X_{401k}$. So by (\ref{H-def1}) we obtain $v\in V(H)$ and in particular the graph $H$ is non-empty.
\end{proof}

Let us derive some use some useful properties of $H$ in order to apply  Lemma \ref{mainlem} to $H$ afterwards.

\begin{lem}\label{H-edge-defect}
$(k-1)v(H)-e(H)\leq 12\vert\C_{\l+1}'\vert$.
\end{lem}
\begin{proof} From (\ref{H-def2}) we obtain
\begin{equation}\label{H-size}
v(H)=n-(\vert X_1\vert+\dots+\vert X_{401k}\vert)-\sum_{D\in \C_1\cup\dots\cup\C_{J'}}\vert D\vert-\sum_{D\in \C_{\l+1}'}\vert D\vert-\sum_{\substack{D\in \C_{J'+1}\cup\dots\cup\C_{\l}\\ \text{uncoloured}}}\vert D\vert.
\end{equation}
On the other hand
$$e(H)\geq e(G)-(\e_G(X_1)+\dots+\e_G(X_{401k}))-\sum_{D\in \C_1\cup\dots\cup\C_{J'}}\e_G(D)-\sum_{D\in \C_{\l+1}'}\e_G(D)-\sum_{\substack{D\in \C_{J'+1}\cup\dots\cup\C_{\l}\\ \text{uncoloured}}}\e_G(D).$$
By condition (iii) in Definition \ref{approcolour} and by Corollary \ref{coro-edges-good-set} this implies
\begin{multline*}
e(H)\geq ((k-1)n-t)-\sum_{i=1}^{401k}((k-1)\vert X_i\vert+y_i^{(\l)})-\sum_{D\in \C_1\cup\dots\cup\C_{J'}}((k-1)\vert D\vert+1)-\sum_{D\in \C_{\l+1}'}((k-1)\vert D\vert+1)\\
-\sum_{\substack{D\in \C_{J'+1}\cup\dots\cup\C_{\l}\\ \text{uncoloured}}}((k-1)\vert D\vert+1).
\end{multline*}
Together with (\ref{H-size}) we get
\begin{multline*}
(k-1)v(H)-e(H)\leq t+(y_1^{(\l)}+\dots+y_{401k}^{(\l)})+(\vert\C_1\vert+\dots+\vert\C_{J'}\vert)+\vert\C_{\l+1}'\vert\\
+\vert\lbrace D\in \C_{J'+1}\cup\dots\cup\C_{\l}\,\vert\, D\text{ uncoloured}\rbrace\vert
\end{multline*}
and by (\ref{y-sum})
$$(k-1)v(H)-e(H)\leq t+(\vert\C_1\vert+\dots+\vert\C_{J'}\vert)+(\vert\C_{J'+1}\vert+\dots+\vert\C_{\l}\vert)+\vert\C_{\l+1}'\vert\leq t+\vert\C_1\vert+\dots+\vert\C_{\l+1}\vert.$$
Using Claim \ref{Jstrichgross} we obtain
$$(k-1)v(H)-e(H)\leq 2^{J'}+(2^{0}+\dots+2^{\l})=2^{J'}+2^{\l+1}-1<3\cdot 2^{\l}=3\vert \C_{\l+1}\vert.$$
By Claim \ref{Clplus1strich} we can assume $\vert \C_{\l+1}\vert<4\vert \C_{\l+1}'\vert$ and this gives $(k-1)v(H)-e(H)\leq 12\vert\C_{\l+1}'\vert$ as desired.
\end{proof}

\begin{claim}\label{H-propC} For every $D\in \C$ we have $D\su V(H)$ or $D\cap V(H)=\es$.
\end{claim}
\begin{proof}If $D\in \C_1\cup\dots\cup\C_{\l+1}$, then by (\ref{H-def1}) clearly $D\cap V(H)=\es$. Otherwise $D$ is disjoint from all elements of $\C_1\cup\dots\cup\C_{\l+1}$. Also, by condition (ii) in Definition \ref{approcolour} the $D$ set is either contained in some $X_i$ or disjoint from $X_1\cup\dots\cup X_{401k}$. By (\ref{H-def1}) in the first case we have $D\cap V(H)=\es$ and in the second case $D\su V(H)$.
\end{proof}

Let $\C_H$ be the collection of those $D\in \C$ with $D\su V(H)$. Note that by (\ref{H-def1}) we have $\C_H\su \C_{\l+2}\cup\dots\cup\C_{J}$.

Now we check that $H$ together with the collection $\C_H$ of subsets of $V(H)$ satisfies the assumptions of Lemma \ref{mainlem}: The elements of $\C_H$ are disjoint, as all elements of $\C$ are disjoint. Each $D\in \C_H$ is non-empty and we have by Corollary \ref{coro-edges-good-set}
$$\e_H(D)\leq \e_G(D)\leq (k-1)\vert D\vert+1.$$
Also, for each $D\in \C_H$ by (\ref{H-def1}) we have $D\in \C_{\l+2}\cup\dots\cup\C_{J}$ and $D$ is disjoint from $X_1\cup\dots\cup X_{401k}$, hence $D$ is uncoloured in the colouring $\varphi$. By condition (vii) in Definition \ref{approcolour} this implies that all neighbours of $D$ in $G$ are also uncoloured (and therefore not in $X_1\cup\dots\cup X_{401k}$). Since $D$ does not have edges to any other member of $\C$, this implies that $H$ contains all neighbours of $D$ in $G$. Thus, for each $v\in D$ we have $\deg_H(v)=\deg_G(v)\geq k$.

Finally, from (\ref{H-def2}) it is clear that
$$v(H)\leq n-(\Vert \C_1\Vert+\dots+\Vert \C_{J'}\Vert)$$
and using (\ref{eqJstrich1}) this gives $v(H)\leq n-\frac{n}{100k}<(1-\eps)n$. Since $G$ does not contain a subgraph on at most $(1-\eps)n$ vertices with minimum degree at least $k$, we can conclude that $H$ does not have a subgraph of minimum degree at least $k$.

Thus, $H$ does indeed satisfy all assumptions of Lemma \ref{mainlem}. By applying this lemma we find a subset $S\su V_{\leq k-1}(H)$ with $V_{\leq k-2}(H)\su S$ and
\begin{equation}\label{ineq-s-sect3-1}
\sum_{s\in S}(k-\deg_H(s))\leq 2((k-1)v(H)-e(H))
\end{equation}
as well as disjoint subsets $B_v\su V(H)$ for each vertex $v\in V_{\leq k-1}(H)\mi S=V_{k-1}(H)\mi S$, such that the statement in Lemma \ref{mainlem} holds.

From (\ref{ineq-s-sect3-1}) and Lemma \ref{H-edge-defect} we obtain
$$\sum_{s\in S}(k-\deg_H(s))\leq 2((k-1)v(H)-e(H))\leq 24\vert\C_{\l+1}'\vert.$$
Since $\deg_H(s)\leq k-1$ for each $s\in S$, this in particular implies $\vert S\vert\leq 24\vert\C_{\l+1}'\vert$. This, in turn, gives
\begin{equation}\label{ineq-s-sect3-2}
\sum_{s\in S}(k+1-\deg_H(s))=\vert S\vert+\sum_{s\in S}(k-\deg_H(s))\leq 48\vert\C_{\l+1}'\vert.
\end{equation}

The next step in the proof is to obtain a colouring of some of the members of $\C_{\l+1}'$ (recall that all members of $\C_{\l+1}'$ are uncoloured in $\varphi$). Afterwards, we extend $\varphi$ to an $(\l+1)$-appropriate colouring of $G$ by using the colouring on $\C_{\l+1}'$ and the statement in Lemma \ref{mainlem}.

Let $S'\su S$ be the set of all $s\in S$ that are adjacent to at least one member of $\C_{\l+1}'$ in the graph $G$.

For each $s\in S'$ we define a set $\C(s)\su\C_{\l+1}'$ as follows: If at least $k+1-\deg_H(s)$ members of $\C_{\l+1}'$ are adjacent to $s$, then we pick any $k+1-\deg_H(s)$ of them to form $\C(s)$.  If less than $k+1-\deg_H(s)$ members of $\C_{\l+1}'$ are adjacent to $s$, let $\C(s)$ consist of all of them. In either case we get $\vert \C(s)\vert\leq k+1-\deg_H(s)\leq k+1$ and each member of $\C(s)$ is adjacent to $s$.

By (\ref{ineq-s-sect3-2}) we have
\begin{equation}\label{ineq-s-sect3-3}
\sum_{s\in S'}\vert \C(s)\vert\leq \sum_{s\in S}(k+1-\deg_H(s))\leq 48\vert\C_{\l+1}'\vert.
\end{equation}
Let us call $D\in \C_{\l+1}'$ \emph{popular}, if $D$ is contained in more than $200$ sets $\C(s)$ for $s\in S'$. Otherwise, let us call $D$ \emph{non-popular}. By (\ref{ineq-s-sect3-3}) the number of popular elements in $\C_{\l+1}'$ is at most
$$\frac{1}{200}\sum_{s\in S'}\vert \C(s)\vert\leq \frac{48}{200}\vert\C_{\l+1}'\vert<\frac{1}{4}\vert\C_{\l+1}'\vert\leq \frac{1}{4}\vert\C_{\l+1}\vert.$$

Our goal is to colour $\C_{\l+1}'$ with the $401k$ colours in such a way that only the popular sets are left uncoloured.

In order to construct the colouring, let us define a list $L(s)$ of up to $k$ colours for each vertex $s\in S'$: If the neighbours of $s$ in $G$ have at most $k$ different colours in the colouring $\varphi$, then let the list $L(s)$ contain all of these colours. If the neighbours of $s$ in $G$ have more than $k$ different colours in the colouring $\varphi$, then we pick any $k$ of these colours for the list $L(s)$.

Now we colour the non-popular elements of $\C_{\l+1}'$ with the $401k$ colours, so that each non-popular element of $\C_{\l+1}'$ receives exactly one colour. Furthermore, for each $s\in S'$, we want to ensure that all the non-popular elements of $\C(s)$ receive distinct colours and none of the elements of $\C(s)$ receives a colour on the list $L(s)$. This is possible by a simple greedy algorithm, since each non-popular $D\in \C_{\l+1}'$ is an element of at most 200 sets $\C(s)$, and each of these sets $\C(s)$ has at most $k$ other elements and for each of these sets $\C(s)$ there are at most $k$ forbidden colours on the list $L(s)$. So each of the at most 200 sets $\C(s)$ can rule out at most $2k$ colours for $D$, and $200\cdot 2k<401k$.

Let us denote this colouring of $\C_{\l+1}'$ by $\psi$. Then for each $s\in S'$, all the non-popular elements of $\C(s)$ have distinct colours in the colouring $\psi$ and none of them has a colour on the list $L(s)$. Furthermore, all non-popular members of $\C_{\l+1}'$ are coloured with exactly one colour in $\psi$. Since $\C_{\l+1}'$ has at most $\frac{1}{4}\vert\C_{\l+1}\vert$ popular elements, this means that at least $\vert\C_{\l+1}'\vert-\frac{1}{4}\vert\C_{\l+1}\vert$ members of $\C_{\l+1}'$ are coloured in $\psi$. For each $i=1,\dots,401k$, let $\Z_i\su \C_{\l+1}'$ consist of all the members of $\C_{\l+1}'$ which are coloured with colour $i$ in $\psi$.

Our goal is to construct an $(\l+1)$-appropriate colouring of $G$. For that we use the colourings $\varphi$ and $\psi$ (recall that $\psi$ is a colouring on $\C_{\l+1}'$ and all of $\C_{\l+1}'$ is uncoloured in $\varphi$), but we may also need to extend the colouring to some vertices in $H$. More precisely, for each of the $401k$ colours we apply the statement in Lemma \ref{mainlem} to a different graph $\tH$ (in each case obtaining a subgraph $\tH'$ with the six properties listed in the lemma), and then also colour the set $V(\tH)\mi V(\tH')\su V(H)$ with the corresponding colour. We then check that together with $\varphi$ and $\psi$ this defines an $(\l+1)$-appropriate colouring of $G$.

In order to apply this plan, consider any of the $401k$ colours. To minimize confusion, let us call this colour red. As before, let $X_{\R}$ be the set of vertices of $G$ coloured red in $\varphi$ and $y_{\R}^{(\l)}$ the number of members of $\C_1\cup\dots\cup\C_{\l}$ that are monochromatically red in $\varphi$. The set $\Z_{\R}\su \C_{\l+1}'$ consists of the red members of $\C_{\l+1}'$ in the colouring $\psi$.

Let us consider the graph
\begin{equation}\label{def-tH-red}
\tH_\R=G-X_\R-\Z_\R.
\end{equation}
Let us now check that this graph has all properties required to act as $\tH$ in Lemma \ref{mainlem}.

\begin{claim}\label{tH-red-prop0}
$\tH_\R$ contains $H$ as a proper induced subgraph.
\end{claim}
\begin{proof}
As $X_\R\su X_1\cup\dots\cup X_{401k}$ and $\Z_\R\su \C_{\l+1}'\su \C_{\l+1}$, it follows from (\ref{H-def1}), that $\tH_\R$ contains $H$ as an induced subgraph.

On the other hand, by condition (v) in Definition \ref{approcolour} the members of $\C_1\cup\dots\cup\C_{J'}$ are all completely uncoloured in $\varphi$ and therefore disjoint from $X_\R$. By $\Z_\R\su \C_{\l+1}'\su \C_{\l+1}$ (and since any two members of $\C$ are disjoint), the members of $\C_1\cup\dots\cup\C_{J'}$ are also disjoint from all members of $\Z_\R$. Hence all members of $\C_1\cup\dots\cup\C_{J'}$ are subsets of $V(\tH_\R)=G-X_\R-\Z_\R$. But by (\ref{H-def1}) they are all disjoint from $V(H)$. This establishes that $H$ must be a proper induced subgraph of $\tH_\R$.
\end{proof}

\begin{claim}\label{tH-red-prop1}
$V_{\leq k-1}(\tH_\R)\su V(H)$.
\end{claim}
\begin{proof} By condition (iv) in Definition \ref{approcolour} the graph $G-X_\R$ has minimum degree at least $k$. So all vertices of $\tH_\R=G-X_\R-\Z_\R$ with degree at most $k-1$ are neighbours of members of $\Z_\R$ in $G$. But the members of $\Z_\R$ do not have any edges to members of $\C\mi \Z_\R$ (as there are no edges between different members of $\C$). Also, every $D\in \Z_\R\su \C_{\l+1}'$ is uncoloured in $\varphi$ and by condition (vii) in Definition \ref{approcolour} this implies that all neighbours of $D$ in $G$ are also uncoloured in $\varphi$. Thus, $D$ does not have any neighbours in $X_1\cup\dots\cup X_{401k}$. Thus, all neighbours of members of $\Z_\R$ in $G$ lie either in $\Z_\R$ itself or in $H=G-(\C_1\cup\dots\cup\C_{\l+1})-(X_1\cup\dots\cup X_{401k})$. Thus, all vertices of $\tH_\R=G-X_\R-\Z_\R$ with degree at most $k-1$ are in $V(H)$.
\end{proof}

\begin{claim}\label{tH-red-prop2}
$V_{\leq k-1}(\tH_\R)\su V_{\leq k-1}(H)\mi S$.
\end{claim}
\begin{proof} By Claim \ref{tH-red-prop1} we already know $V_{\leq k-1}(\tH_\R)\su V(H)$. Since $\tH_\R$ contains $H$ as an induced subgraph by Claim \ref{tH-red-prop0}, this implies $V_{\leq k-1}(\tH_\R)\su V_{\leq k-1}(H)$. So it remains to show that every vertex $s\in S\su V_{\leq k-1}(H)$ has degree at least $k$ in $\tH_\R$.

If $s$ is not adjacent to any member of $\Z_\R$, then the degree of $s$ in $\tH_\R=G-X_\R-\Z_\R$ is equal to the degree of $s$ in $G-X_\R$, which is at least $k$ by condition (iv) in Definition \ref{approcolour}.

If $s\not\in S'$, then $s$ is not adjacent to any member of $\C_{\l+1}'$ and in particular not to any member of $\Z_\R$, so we are done by the previous observation.

So let us now assume that $s\in S'$ and that $s$ is adjacent to some member $D$ of $\Z_\R$.

If $\vert \C(s)\vert=k+1-\deg_H(s)$, then at least $k-\deg_H(s)$ members of $\C(s)$ are not red in $\psi$ (since at most one member of $\C(s)$ is red). Since $\C(s)\su \C_{\l+1}'$ all of its members are disjoint from $X_1\cup\dots\cup X_{401k}$ and in particular from $X_\R$. Thus, all of the at least $k-\deg_H(s)$ non-red members of $\C(s)\su \C_{\l+1}'$ are present in $\tH_\R=G-X_\R-\Z_\R$ and they are all neighbours of $s$. Furthermore, $s$ has $\deg_H(s)$ additional neighbours in $V(H)\su V(\tH_\R)$, so in total $s$ has degree at least $(k-\deg_H(s))+\deg_H(s)=k$ in $\tH_\R$.

Thus, we can assume $\vert \C(s)\vert<k+1-\deg_H(s)$. By definition of $\C(s)$, this means that $\C(s)$ contains all the neighbours of $s$ in $\C_{\l+1}'$. Also recall our assumption that $s$ is adjacent to some member $D$ of $\Z_\R\su \C_{\l+1}'$. Then $D\in \C(s)$ and $D$ is coloured red in $\psi$. By the properties of $\psi$, the colour red is not on the list $L(s)$ and $D$ is the only red element in $\C(s)$. Thus, $D$ is the only neighbour of $s$ in $\Z_\R$.

If $\vert L(s)\vert=k$, then $s$ has neighbours in $G$ in $k$ different non-red colours. All these vertices lie in $\tH_\R=G-X_\R-\Z_\R$. Thus, $s$ has degree at least $k$ in $\tH_\R$.

So we may assume $\vert L(s)\vert<k$. Then from the definition of $L(s)$ we obtain that $L(s)$ contains all colours of the neighbours of $s$ in $G$. Since red is not on the list $L(s)$, we can conclude that $s$ has no neighbours in $G$ inside the set $X_\R$. Also, recall that $D$ is the only neighbour of $s$ in $\Z_\R$. Hence the degree of $s$ in $\tH_\R=G-X_\R-\Z_\R$ is the same as in $G-D$, which is at least $k$ by Claim \ref{remove-good-set}. This finishes the proof of Claim \ref{tH-red-prop2}.
\end{proof}

\begin{claim}\label{tH-red-prop4}
If $v\in V_{\leq k-1}(\tH_\R)$, then $v\in V_{k-1}(H)\mi S$, $v$ has at least one neighbour in $\C_{\l+1}'$ and all neighbours of $v$ in $\C_{\l+1}'$ are red in $\psi$.
\end{claim}
\begin{proof}
By Claim \ref{tH-red-prop2} we know that $v\in V_{\leq k-1}(H)\mi S$. Since $V_{\leq k-2}(H)\su S$, we have $V_{\leq k-1}(H)\mi S=V_{k-1}(H)\mi S$. Thus, $v\in V_{k-1}(H)\mi S$. So $v$ has degree $k-1$ in $H$. Since $v$ by assumption has degree at most $k-1$ in $\tH_\R$ and $\tH_\R$ contains $H$ as an induced subgraph, we can conclude that $v$ also has degree $k-1$ in $\tH_\R$ and it does not have any neighbours in $V(\tH_\R)\mi V(H)$.

Since $v\in \tH_\R=G-X_\R-\Z_\R$ we have $v\not\in X_\R$ and by condition (iv) in Definition \ref{approcolour} we know that $v$ has degree at least $k$ in $G-X_\R$. But since $v$ has degree $k-1$ in $\tH_\R=G-X_\R-\Z_\R$, this implies that $v$ has at least one neighbour in $\Z_\R\su \C_{\l+1}'$.

Note that all members of $\C_{\l+1}'\mi \Z_\R$ are subsets of $V(\tH_\R)\mi V(H)$: Each $D\in \C_{\l+1}'\mi \Z_\R$ is disjoint from $X_\R$ (because it is uncoloured in $\varphi$) and is therefore a subset of $V(\tH_\R)$, but by (\ref{H-def2}) it is clearly disjoint from $V(H)$. By the argument above, $v$ does not have any neighbours in $V(\tH_\R)\mi V(H)$. Hence $v$ does not have any neighbours in $\C_{\l+1}'\mi \Z_\R$ and therefore all the neighbours of $v$ in $\C_{\l+1}'$ belong to $\Z_\R$, i.e.\ they are red in $\psi$.
\end{proof}

\begin{claim}\label{tH-red-prop3}
No vertex in $V(\tH_\R)\mi V(H)$ is adjacent to any member of $\C_H$.
\end{claim}
\begin{proof}
Let $D\in \C_H\su \C_{\l+2}\cup\dots\cup\C_{J}$, then $D$ does not have any edges to members of $\C_1\cup\dots\cup\C_{\l+1}$. Furthermore by (\ref{H-def1}) we know that $D\su V(H)$ is disjoint from $X_1\cup\dots\cup X_{401k}$, which means that $D$ is uncoloured in $\varphi$. So by condition (vii) in Definition \ref{approcolour} all neighbours of $D$ in $G$ are also uncoloured in $\varphi$. Hence $D$ has no neighbours in $X_1\cup\dots\cup X_{401k}$. Thus,
$$H=G-(\C_1\cup\dots\cup\C_{\l+1})-(X_1\cup\dots\cup X_{401k})$$
contains all neighbours of $D$ in $G$ and in particular all neighbours of $D$ in $V(\tH_\R)$. Hence no vertex in $V(\tH_\R)\mi V(H)$ is adjacent to $D$.
\end{proof}

By Claim \ref{tH-red-prop0}, Claim \ref{tH-red-prop2} and Claim \ref{tH-red-prop3} the graph $\tH_\R$ satisfies all properties to act as $\tH$ in Lemma \ref{mainlem}. So by the conclusion of the lemma, there is an induced subgraph $\tH_\R'$ of $\tH_\R$ with all the properties listed in Lemma \ref{mainlem}.

Set
\begin{equation}\label{Xstrich-red}
X_\R'=V(\tH_\R)\mi V(\tH_\R').
\end{equation}
This is the set of vertices we want to colour red in addition to the red vertices in $\varphi$ and the red members of $\C_{\l+1}'$ in $\psi$. But first, we need to establish some properties of the set $X_\R'$.

By property (c) in Lemma \ref{mainlem} we have
\begin{equation}\label{Xstrich-red2}
X_\R'=V(\tH_\R)\mi V(\tH_\R')\su \bigcup_{v\in V_{\leq k-1}(\tH_\R)}B_v\su V(H).
\end{equation}

\begin{claim}\label{Xstrich-red6}
For each $D\in \C$ we either have $D\su X_\R'$ or $D\cap X_\R'=\es$.
\end{claim}
\begin{proof} By Claim \ref{H-propC} we have $D\su V(H)$ or $D\cap V(H)=\es$. In the latter case we have $D\cap X_\R'=\es$ by (\ref{Xstrich-red2}). In the former case we have $D\in \C_H$ and hence property (e) in Lemma \ref{mainlem} gives $D\su V(\tH_\R')$ or $D\cap V(\tH_\R')=\es$. Since $D\su V(H)$ implies $D\su V(\tH_\R)$, we obtain $D\cap (V(\tH_\R)\mi V(\tH_\R'))=\es$ or $D\su (V(\tH_\R)\mi V(\tH_\R'))$. This proves the claim as $X_\R'=V(\tH_\R)\mi V(\tH_\R')$.
\end{proof}

\begin{claim}\label{Xstrich-red5}
$\e_{G-X_\R-\Z_\R}(X_\R')\leq (k-1)\vert X_\R'\vert$.
\end{claim}
\begin{proof} By property (b) in Lemma \ref{mainlem} we have
$$(k-1)v(\tH_\R')-e(\tH_\R')\leq (k-1)v(\tH_\R)-e(\tH_\R),$$
that is
$$e(\tH_\R)-e(\tH_\R')\leq (k-1)(v(\tH_\R)-v(\tH_\R')).$$
So recalling $\tH_\R=G-X_\R-\Z_\R$ and $X_\R'=V(\tH_\R)\mi V(\tH_\R')$ we obtain
$$\e_{G-X_\R-\Z_\R}(X_\R')=\e_{\tH_\R}(X_\R')=e(\tH_\R)-e(\tH_\R')\leq (k-1)(v(\tH_\R)-v(\tH_\R'))=(k-1)\vert X_\R'\vert,$$
which concludes the proof of the claim.
\end{proof}

\begin{claim}\label{Xstrich-red7}
The graph $G-X_\R-\Z_\R-X_\R'$ has minimum degree at least $k$.
\end{claim}
\begin{proof} By (\ref{def-tH-red}) and (\ref{Xstrich-red}) we have
$$G-X_\R-\Z_\R-X_\R'=\tH_\R-X_\R'=\tH_\R'$$
and by property (a) in Lemma \ref{mainlem} this is a (non-empty) graph with minimum degree at least $k$.
\end{proof}

\begin{claim}\label{Xstrich-red8}
We have
$$X_\R'\su \bigcup_{v}B_v,$$
where the union is taken over all $v\in V_{k-1}(H)\mi S$, which have at least one neighbour in $\C_{\l+1}'$ and such that all neighbours of $v$ in $\C_{\l+1}'$ are red in $\psi$.
\end{claim}
\begin{proof} This is a direct consequence of (\ref{Xstrich-red2}) and Claim \ref{tH-red-prop4}.
\end{proof}

\begin{claim}\label{Xstrich-red9}
The set $X_\R'$ is disjoint from all members of $\C_1\cup\dots\cup\C_{J'}$.
\end{claim}
\begin{proof} By (\ref{Xstrich-red2}) we have $X_\R'\su V(H)$, but by (\ref{H-def2}) all members of $\C_1\cup\dots\cup\C_{J'}$ are disjoint from $V(H)$.
\end{proof}

\begin{claim}\label{Xstrich-red10}
If $D\in \C_H$ and $D\cap X_\R'=\es$, then $D$ has no neighbours in $X_\R'$.
\end{claim}
\begin{proof} From $D\in \C_H$ we know $D\su V(H)\su V(\tH_\R)$ and because of $D\cap X_\R'=\es$ and $X_\R'=V(\tH_\R)\mi V(\tH_\R')$, this implies $D\su V(\tH_\R')$. Then by property (f) in Lemma \ref{mainlem} the set $D$  has no neighbours in $V(\tH_\R)\mi V(\tH_\R')=X_\R'$.
\end{proof}

Finally, let us define the desired $(\l+1)$-appropriate colouring of $G$. All the above considerations hold when red is any of the $401k$ colours (in the arguments above it is just called `red' instead of `colour $i$' to make notation less confusing, since there are already number indices for the sets $\C_j$). So for each $i=1,\dots, 401k$ we can take red to be colour $i$ and apply the above arguments. Then for each $i=1,\dots, 401k$ we get a set $X_i'\su V(H)$ with all the above properties (where we replace each index `red' by $i$ and each word `red' by `colour $i$').

We define a colouring $\rho$ of $G$ as follows: Start with the colouring $\varphi$. Now colour all non-popular members of $\C_{\l+1}'$ according to the colouring $\psi$ (recall that all members of $\C_{\l+1}'$ are uncoloured in $\varphi$). Finally, for each $i=1,\dots, 401k$ colour all vertices in the set $X_i'$ with colour $i$. 

Let us now check that the colouring $\rho$ is $(\l+1)$-appropriate:
\begin{itemize}
\item[(i)] We need to check that each vertex has at most one colour in the colouring $\rho$. We know (since the colouring $\varphi$ is $\l$-appropriate), that every vertex has at most one colour in the colouring $\varphi$. When colouring the members of $\C_{\l+1}'$ according to the colouring $\psi$, we only colour vertices that are uncoloured in $\varphi$ (because all members of $\C_{\l+1}'$ are by definition uncoloured in $\varphi$), so after applying $\psi$ still every vertex has at most one colour.

Recall that $X_i'\su V(H)$ by (\ref{Xstrich-red2}) and that all vertices of $H$ are uncoloured in $\varphi$ by (\ref{H-def1}). Furthermore $\psi$ only coloured members of $\C_{\l+1}'$ and by (\ref{H-def2}) these are all disjoint from $V(H)$. Hence the sets $X_i'$ only consist of vertices that have not been coloured yet by $\varphi$ and $\psi$.

Thus, it remains to check that the sets $X_i'$ for $i=1,\dots,401k$ are disjoint. For each $i=1,\dots,401k$ we have by Claim \ref{Xstrich-red8}
\begin{equation}\label{Xstrich-i-1}
X_i'\su \bigcup_{v}B_v,
\end{equation}
where the union is taken over all $v\in V_{k-1}(H)\mi S$, which have at least one neighbour in $\C_{\l+1}'$ and such that all neighbours of $v$ in $\C_{\l+1}'$ have colour $i$ in $\psi$. Note that for each vertex $v\in V_{k-1}(H)\mi S$ with at least one neighbour in $\C_{\l+1}'$, there is at most one colour $i$ such that all neighbours of $v$ in $\C_{\l+1}'$ have colour $i$ in $\psi$. Hence for each vertex $v\in V_{k-1}(H)\mi S$ the corresponding set $B_v$ appears in (\ref{Xstrich-i-1}) for at most one $i$. Since the sets $B_v$ for $v\in V_{k-1}(H)\mi S$ are disjoint, the right-hand-sides of (\ref{Xstrich-i-1}) for $i=1,\dots,401k$ are disjoint. Hence the sets $X_i'$ for $i=1,\dots,401k$ are disjoint.
\item[(ii)] We need to show that each $D\in \C$ is either monochromatic or completely uncoloured in $\rho$. We already know that this is true in $\varphi$ (since $\varphi$ is $\l$-appropriate), and $\psi$ only colours entire members of $\C_{\l+1}'$ (recall that all members of $\C$ are disjoint). So it suffices to show that for each $D\in \C$ and each $i=1,\dots,401k$ we have $D\cap X_i'=\es$ or $D\su X_i'$. This is true by Claim \ref{Xstrich-red6}.
\item[(iii)] For $i=1,\dots,401k$, the set of vertices having colour $i$ in $\rho$ is
$$X_i\cup X_i'\cup\bigcup_{D\in \Z_i}D,$$
and we remember that this is a disjoint union.

Let $y_i^{(\l+1)}$ be the number of members of $\C_1\cup\dots\cup\C_{\l+1}$ that are coloured in colour $i$ in $\rho$. This includes the $y_i^{(\l)}$ members of $\C_1\cup\dots\cup\C_\l$ with colour $i$ in $\varphi$ and the members of $\Z_i\su \C_{\l+1}'\su \C_{\l+1}$. Thus,
\begin{equation}\label{Xstrich-i-2}
y_i^{(\l+1)}\geq y_i^{(\l)}+\vert \Z_i\vert.
\end{equation}
Now, for the set $X_i\cup X_i'\cup\bigcup_{D\in \Z_i}D$ of vertices having colour $i$ in $\rho$ we obtain
\begin{multline*}
\e_G\left(X_i\cup X_i'\cup\bigcup_{D\in \Z_i}D\right)=\e_G\left(X_i\cup\bigcup_{D\in \Z_i}D\right)+\e_{G-X_i-\Z_i}(X_i')\\
\leq \e_G(X_i)+\sum_{D\in \Z_i} \e_G(D)+\e_{G-X_i-\Z_i}(X_i').
\end{multline*}
We know that $\e_G(X_i)\leq (k-1)\vert X_i\vert+y_i^{(\l)}$ (as $\varphi$ is an $\l$-appropriate colouring), $\e_G(D)\leq (k-1)\vert D\vert+1$ for each $D\in \Z_i\su \C$ by Corollary \ref{coro-edges-good-set} and $\e_{G-X_i-\Z_i}(X_i')\leq (k-1)\vert X_i'\vert$ by Claim \ref{Xstrich-red5}. Plugging all of that in, we obtain
\begin{multline*}
\e_G\left(X_i\cup X_i'\cup\bigcup_{D\in \Z_i}D\right)\leq (k-1)\vert X_i\vert+y_i^{(\l)}+\sum_{D\in \Z_i} ((k-1)\vert D\vert+1)+(k-1)\vert X_i'\vert\\
=(k-1)\left(\vert X_i\vert+\vert X_i'\vert+\sum_{D\in \Z_i}\vert D\vert\right)+y_i^{(\l)}+\vert \Z_i\vert\leq (k-1)\left\vert X_i\cup X_i'\cup\bigcup_{D\in \Z_i}D\right\vert+y_i^{(\l+1)},
\end{multline*}
where the last inequality follows from (\ref{Xstrich-i-2}).
\item[(iv)] For each $i=1,\dots,401k$ the graph $G-X_i-\Z_i-X_i'$ has minimum degree at least $k$ by Claim \ref{Xstrich-red7}.
\item[(v)] The members of $\C_1\cup\dots\cup\C_{J'}$ are all uncoloured in $\varphi$, since $\varphi$ is $\l$-appropriate. Furthermore, they do not get coloured by $\psi$, since $\psi$ only colours members of $\C_{\l+1}'\su \C_{\l+1}$. By Claim \ref{Xstrich-red9} the members of $\C_1\cup\dots\cup\C_{J'}$ are also disjoint from all $X_i'$. Thus, the members of $\C_1\cup\dots\cup\C_{J'}$ are uncoloured in $\rho$.
\item[(vi)] For every $J'+1\leq j\leq \l$ the number of members of $\C_j$ that are uncoloured in $\rho$ is at most the number of members of $\C_j$ that are uncoloured in $\varphi$ (actually, one can check that these two numbers are equal, but this is not necessary for the argument). This latter number is at most $\frac{1}{4}\vert \C_j\vert$, since $\varphi$ is $\l$-appropriate. Hence in $\rho$ the number of uncoloured members of $\C_j$ is at most $\frac{1}{4}\vert \C_j\vert$ for each $J'+1\leq j\leq \l$.

For $j=\l+1$ note that by definition of $\C_{\l+1}'$ all members of $\C_{\l+1}\mi \C_{\l+1}'$ are coloured in $\varphi$ and hence also in $\rho$. Furthermore at least $\vert\C_{\l+1}'\vert-\frac{1}{4}\vert\C_{\l+1}\vert$ members of $\C_{\l+1}'\su \C_{\l+1}$ are coloured in $\psi$ and hence also in $\rho$. So all in all at least
$$\vert \C_{\l+1}\mi \C_{\l+1}'\vert+\vert\C_{\l+1}'\vert-\frac{1}{4}\vert\C_{\l+1}\vert=(\vert\C_{\l+1}\vert-\vert\C_{\l+1}'\vert)+(\vert\C_{\l+1}'\vert-\frac{1}{4}\vert\C_{\l+1}\vert)=\frac{3}{4}\vert\C_{\l+1}\vert$$
members of $\C_{\l+1}$ are coloured in $\rho$. So the number of members of $\C_{\l+1}$ that are uncoloured in $\rho$ is at most $\frac{1}{4}\vert \C_{\l+1}\vert$.
\item[(vii)] Let $D\in \C_{\l+2}\cup\dots\cup\C_{J}$ be uncoloured in $\rho$. We have to show that all of its neighbours $v\in V(G)$ are also uncoloured in $\rho$.

First, note that $D$ is disjoint from all members of $\C_{1}\cup\dots\cup\C_{\l+1}$. Since $D$ is uncoloured in $\rho$, it is also uncoloured in $\varphi$ and hence disjoint from $X_1\cup\dots\cup X_{401k}$. By (\ref{H-def1}) we can conclude that $D\su V(H)$, i.e.\ $D\in \C_H$. Since $D$ is uncoloured in $\rho$, it is disjoint from $X_1',\dots ,X_{401k}'$. Now Claim \ref{Xstrich-red10} implies that $D$ has no neighbours in $X_i'$ for any $i=1,\dots,401k$.

Trivially $D\in \C_{\l+1}\cup\dots\cup\C_{J}$ and since $\varphi$ is $\l$-appropriate, all neighbours of $D$ are uncoloured in $\varphi$. When applying the colouring $\psi$, we only colour members of $\C_{\l+1}'$, but $D\in \C_{\l+2}\cup\dots\cup\C_{J}$ has no neighbours in any member of $\C_{\l+1}'$. Hence all neighbours of $D$ are still uncoloured after applying $\psi$. Since $D$ has no neighbours in $X_i'$ for any $i=1,\dots,401k$, we can conclude that all neighbours of $D$ are uncoloured in $\rho$.
\end{itemize}

Hence the colouring $\rho$ is indeed $(\l+1)$-appropriate. This finishes the induction step and the proof of Lemma \ref{colour}.

\begin{remark} The reader might notice that we do not use property (d) from Lemma \ref{mainlem} for the proof of Lemma \ref{colour}. However, we prove Lemma \ref{mainlem} in Section \ref{sect5} by induction and we need property (d) in order keep the induction going.
\end{remark}
\section{Preparations for the proof of Lemma \ref{mainlem}} \label{sect4}
In this section, let $H$ be a graph and $\C_H$ be a collection of disjoint non-empty subsets of $V(H)$ such that for each $D\in\C_H$ we have $\e_H(D)\leq (k-1)\vert D\vert+1$ and $\deg_H(v)\geq k$ for each $v\in D$. Note that this means that all $D\in\C_H$ are disjoint from $V_{\leq k-1}(H)$.

The goal of this section is to introduce the shadow $\sh_H(w)$ of a vertex $w\in V_{\leq k-1}(H)$ and establish several useful properties of it. These play an important role in the proof of Lemma \ref{mainlem} in Section \ref{sect5}.

\begin{defi}\label{defi-shadow}For a vertex $w\in V_{\leq k-1}(H)$ define the shadow $\sh_H(w)$ of $w$ in $H$ as the minimal subset $Y\su V(H)$ with the following four properties:
\begin{itemize}
\item[(I)] $w\in Y$.
\item[(II)] For each $D\in \C_H$ either $D\su Y$ or $D\cap Y=\es$.
\item[(III)] If $v\in V(H)\mi Y$ is adjacent to a vertex in $Y$, then $\deg_{H-Y}(v)\geq k$.
\item[(IV)] If $D\in \C_H$ is adjacent to a vertex in $Y$, then $D\su Y$.
\end{itemize}
\end{defi}

First, let us check that there is indeed a unique minimal set with the properties (I) to (IV). Note that $Y=V(H)$ satisfies (I) to (IV). So it suffices to show that if $Y_1$ and $Y_2$ both satisfy these properties, then $Y_1\cap Y_2$ does as well:
\begin{itemize}
\item[(I)] Clearly $w\in Y_1\cap Y_2$.
\item[(II)] If $D\cap Y_i=\es$ for $i=1$ or $i=2$, then $D\cap (Y_1\cap Y_2)=\es$. Otherwise $D\su Y_1$ and $D\su Y_2$ and hence $D\su Y_1\cap Y_2$.
\item[(III)] Let $v\not\in Y_1\cap Y_2$ be adjacent to a vertex in $Y_1\cap Y_2$. Then $v\not\in Y_i$ for $i=1$ or for $i=2$. Let us assume without loss of generality that $v\not\in Y_1$. Since $v$ is adjacent to a vertex in $Y_1$, we have $\deg_{H-Y_1}(v)\geq k$ and hence
$$\deg_{H-(Y_1\cap Y_2)}(v)\geq \deg_{H-Y_1}(v)\geq k.$$
\item[(IV)] If $D\in \C_H$ is adjacent to a vertex in $Y_1\cap Y_2$, then it is both adjacent to a vertex in $Y_1$ and to a vertex in $Y_2$. Hence $D\su Y_1$ and $D\su Y_2$, so $D\su Y_1\cap Y_2$.
\end{itemize}
So there is indeed a unique minimal set $Y\su V(H)$ with the properties (I) to (IV) and Definition \ref{defi-shadow} makes sense.

We now describe a procedure to determine the shadow $\sh_H(w)$ of a vertex $w\in V_{\leq k-1}(H)$.

\begin{proc}\label{proc-shadow}For a vertex $w\in V_{\leq k-1}(H)$ consider the following algorithm during which $Y$ is always a subset of $V(H)$. In the beginning we set $Y=\lbrace w\rbrace$. As long as possible, we perform steps of the following form (if we have multiple options, we may choose either of the available options):
\begin{itemize}
\item If there is a vertex $v\not\in Y$ with $v\not\in D$ for all $D\in \C_H$, such that $v$ is adjacent to a vertex in $Y$ and $\deg_{H-Y}(v)\leq k-1$, then we are allowed to add $v$ to $Y$.
\item If there is a $D\in \C_H$ with $D\cap Y=\es$, such that $D$ is adjacent to a vertex in $Y$, then we are allowed to add all of $D$ to $Y$.
\end{itemize}
We terminate when we cannot perform any of these two steps.
\end{proc}

It is clear that this procedure must eventually terminate, because $Y\su V(H)$ becomes larger in every step and $V(H)$ is a finite set. Next, we show that when the procedure terminates, we always have $Y=\sh_H(w)$, independently of the choices we made during the procedure (in case we had multiple allowed steps to choose from).

\begin{claim}\label{proc-prop1} During Procedure \ref{proc-shadow} we always have $Y\su \sh_H(w)$.
\end{claim}
\begin{proof} In the beginning we have $Y=\lbrace w\rbrace$ and $\lbrace w\rbrace\su \sh_H(w)$ by property (I). It remains to show that if $Y\su \sh_H(w)$ and we perform one of the operations in Procedure \ref{proc-shadow}, then the resulting set is still a subset of $\sh_H(w)$.
\begin{itemize}
\item Let $v\not\in Y$ be a vertex that is adjacent to a vertex in $Y\su \sh_H(w)$ and $\deg_{H-Y}(v)\leq k-1$. If $v\not\in \sh_H(w)$, then $v$ is adjacent to a vertex in $\sh_H(w)$ and
$$\deg_{H-\sh_H(w)}(v)\leq \deg_{H-Y}(v)\leq k-1.$$
This would be a contradiction to $\sh_H(w)$ having property (III). Hence we must have $v\in \sh_H(w)$ and hence $Y\cup \lbrace v\rbrace\su \sh_H(w)$.
\item Let $D\in \C_H$ with $D\cap Y=\es$ be adjacent to a vertex in $Y\su \sh_H(w)$, then $D$ is also adjacent to a vertex in $\sh_H(w)$ and by property (IV) of $\sh_H(w)$ we can conclude that $D\su \sh_H(w)$. Hence $Y\cup D\su \sh_H(w)$.
\end{itemize}
This finishes the proof of the claim.\end{proof}

\begin{claim}\label{proc-prop2} During Procedure \ref{proc-shadow} the set $Y$ always satisfies the properties (I) and (II) in Definition \ref{defi-shadow}.
\end{claim}
\begin{proof}For property (I), this is clear as we start with the set $Y=\lbrace w\rbrace$. For property (II), note that $w\in V_{\leq k-1}(H)$ implies $w\not\in D$ for all $D\in \C_H$. During the procedure we only add complete sets $D\in C_H$ (recall that these sets are all disjoint) or vertices that are not contained in any $D\in \C_H$. Thus, $Y$ always satisfies property (II) as well.\end{proof}

\begin{claim}\label{proc-prop3} When Procedure \ref{proc-shadow} terminates, the set $Y$ satisfies the properties (III) and (IV) in Definition \ref{defi-shadow}.
\end{claim}
\begin{proof} Let us start with proving property (IV). Let $D\in \C_H$ be adjacent to a vertex in $Y$, we need to show that $D\su Y$. By Claim \ref{proc-prop2} the set $Y$ satisfies property (II) and therefore $D\su Y$ or $D\cap Y=\es$. If $D\cap Y=\es$, then we could perform the second step in Procedure \ref{proc-shadow} and add $D$ to $Y$. This is a contradiction to the assumption that the procedure already terminated, hence $D\su Y$ as desired.

Now we prove property (III). Let $v\in V(H)\mi Y$ be a vertex that is adjacent to a vertex in $Y$. We have to prove $\deg_{H-Y}(v)\geq k$. Suppose the contrary, i.e.\ $\deg_{H-Y}(v)\leq k-1$. If $v\in D$ for some $D\in \C_H$, then $D$ is adjacent to a vertex in $Y$ as well. But then by property (IV), which we already proved, we would have $D\su Y$ and in particular $v\in Y$, a contradiction. Hence $v\not\in D$ for all $D\in \C_H$. But then we can perform the first step in Procedure \ref{proc-shadow} and add $v$ to $Y$. This is again contradiction to the assumption that the procedure already terminated, hence $\deg_{H-Y}(v)\geq k$.
\end{proof}

\begin{claim}\label{proc-prop4} When Procedure \ref{proc-shadow} terminates, we have $Y=\sh_H(w)$.
\end{claim}
\begin{proof} By Claim \ref{proc-prop1} we have $Y\su \sh_H(w)$. On the other hand, by Claim \ref{proc-prop2} and Claim \ref{proc-prop3} the set $Y$ satisfies the properties (I) to (IV) in Definition \ref{defi-shadow}. Since $\sh_H(w)$ is the unique minimal set satisfying these properties, we have $\sh_H(w)\su Y$. All in all this proves $Y=\sh_H(w)$.
\end{proof}

So Procedure \ref{proc-shadow} indeed determines the shadow $\sh_H(w)$ of $w\in V_{\leq k-1}(H)$ in $H$. We can use this to establish the following important properties of the shadow $\sh_H(w)$, which are used in the proof of Lemma \ref{mainlem} in Section \ref{sect5}.

\begin{lem}\label{proc-prop5} Let $w\in V_{\leq k-1}(H)$. Then during Procedure \ref{proc-shadow} the quantity $(k-1)\vert Y\vert-\e_H(Y)$ is monotone increasing. In the beginning for $Y=\lbrace w\rbrace$ the quantity is non-negative.
\end{lem}
\begin{proof}
In the beginning for $Y=\lbrace w\rbrace$ we have 
$$(k-1)\vert \lbrace w\rbrace\vert-\e_H(\lbrace w\rbrace)=(k-1)-\deg_H(w)\geq 0.$$
Now let us prove that the quantity is monotone increasing:
\begin{itemize}
\item Let $v\not\in Y$ be a vertex with $\deg_{H-Y}(v)\leq k-1$. Then
$$\e_H(Y\cup \lbrace v\rbrace)=\e_H(Y)+\e_{H-Y}(\lbrace v\rbrace)=\e_H(Y)+\deg_{H-Y}(v)\leq \e_H(Y)+(k-1)$$
and therefore
$$(k-1)\vert Y\cup \lbrace v\rbrace\vert-\e_H(Y\cup \lbrace v\rbrace)\geq (k-1)(\vert Y\vert+1)-(\e_H(Y)+(k-1))=(k-1)\vert Y\vert-\e_H(Y).$$
\item Let $D\in \C_H$ be adjacent to a vertex in $Y$ and $D\cap Y=\es$. Then $\e_{H-Y}(D)\leq \e_{H}(D)-1$, because at least one of the edges of $H$ that are incident with a vertex in $D$ leads to a vertex in $Y$ and does therefore not exist in the graph $H-Y$. Furthermore recall that $\e_H(D)\leq (k-1)\vert D\vert+1$ by the assumptions in the beginning of Section \ref{sect4}. Together this gives $\e_{H-Y}(D)\leq (k-1)\vert D\vert$. Now
$$\e_H(Y\cup D)=\e_H(Y)+\e_{H-Y}(D)\leq \e_H(Y)+(k-1)\vert D\vert$$
and therefore
$$(k-1)\vert Y\cup D\vert-\e_H(Y\cup D)\geq (k-1)(\vert Y\vert+\vert D\vert)-(\e_H(Y)+(k-1)\vert D\vert)=(k-1)\vert Y\vert-\e_H(Y).$$
\end{itemize}
\end{proof}

\begin{coro}\label{shadow-prop1} For every $w\in V_{\leq k-1}(H)$ we have $\e_H(\sh_H(w))\leq (k-1)\vert \sh_H(w)\vert$.
\end{coro}
\begin{proof}This is an immediate corollary of Claim \ref{proc-prop4} and Lemma \ref{proc-prop5}.
\end{proof}

\begin{coro}\label{shadow-prop2} Let $w\in V_{\leq k-1}(H)$ with $\e_H(\sh_H(w))= (k-1)\vert \sh_H(w)\vert$. Then in each step during the Procedure \ref{proc-shadow} we have $\e_H(Y)=(k-1)\vert Y\vert$, i.e.\ the quantity $(k-1)\vert Y\vert-\e_H(Y)$ is constantly zero throughout the procedure.
\end{coro}
\begin{proof}This is also an immediate corollary of Claim \ref{proc-prop4} and Lemma \ref{proc-prop5}.
\end{proof}

\begin{lem}\label{shadow-prop3} For every $w\in V_{\leq k-1}(H)$ we have
$$\sum_{\substack{s\in \sh_H(w)\\ \deg_H(s)\leq k-1}}(k-\deg_H(s))\leq (k-1)\vert \sh_H(w)\vert-\e_H(\sh_H(w))+1.$$
\end{lem}
\begin{proof}Let us consider Procedure \ref{proc-shadow} for determining $\sh_H(w)$. If we have multiple options, let us fix some specific choices, so that we obtain some fixed procedure starting with $Y=\lbrace w\rbrace$ and arriving at $Y=\sh_H(w)$. Let us consider the quantity $(k-1)\vert Y\vert-\e_H(Y)$, which is by Lemma \ref{proc-prop5} monotone increasing throughout the procedure.

Consider any $s\in \sh_H(w)$ with $\deg_H(s)\leq k-1$ and $s\neq w$. Since every $D\in \C_H$ is disjoint from $V_{\leq k-1}(H)$ and $s\in V_{\leq k-1}(H)$, the vertex $s$ is not contained in any $D\in \C_H$. So in order to become part of the set $Y$, there must be a step in the procedure where we add precisely the vertex $s$. Let $Y_s$ be the set $Y$ just before this step, then after the step the set $Y$ becomes $Y_s\cup \lbrace s\rbrace$. Note that in order to be allowed to perform this step, the vertex $s$ must be adjacent to some vertex in $Y_s$. Hence $\deg_{H-Y_s}(s)\leq \deg_{H}(s)-1$. Thus,
$$\e_H(Y_s\cup \lbrace s\rbrace)=\e_H(Y_s)+\e_{H-Y_s}(\lbrace s\rbrace)=\e_H(Y_s)+\deg_{H-Y_s}(s)\leq \e_H(Y_s)+(\deg_{H}(s)-1)$$
and therefore
\begin{multline*}
(k-1)\vert Y_s\cup \lbrace s\rbrace\vert-\e_H(Y_s\cup \lbrace s\rbrace)\geq (k-1)(\vert Y_s\vert+1)-(\e_H(Y_s)+\deg_{H}(s)-1)\\
=((k-1)\vert Y_s\vert-\e_H(Y_s))+k-\deg_{H}(s).
\end{multline*}
Hence the quantity $(k-1)\vert Y\vert-\e_H(Y)$ increases by at least $k-\deg_{H}(s)$ from $Y=Y_s$ to $Y=Y_s\cup \lbrace s\rbrace$.

Applying this argument for all $s\in \sh_H(w)$ with $\deg_H(s)\leq k-1$ and $s\neq w$, we can conclude that during the procedure the quantity $(k-1)\vert Y\vert-\e_H(Y)$ increases by at least
$$\sum_{\substack{s\in \sh_H(w)\\ \deg_H(s)\leq k-1\\s\neq w}}(k-\deg_H(s)).$$

In the beginning for $Y=\lbrace w\rbrace$ the quantity $(k-1)\vert Y\vert-\e_H(Y)$ equals
$$(k-1)\vert \lbrace w\rbrace\vert-\e_H(\lbrace w\rbrace)=(k-1)-\deg_H(w).$$

Hence when the process terminates the quantity $(k-1)\vert Y\vert-\e_H(Y)$ must be at least
$$(k-1)-\deg_H(w)+\sum_{\substack{s\in \sh_H(w)\\ \deg_H(s)\leq k-1\\s\neq w}}(k-\deg_H(s))=\sum_{\substack{s\in \sh_H(w)\\ \deg_H(s)\leq k-1}}(k-\deg_H(s))-1.$$
But by Claim \ref{proc-prop4} at the termination point we have $Y=\sh_H(w)$. Thus,
$$(k-1)\vert\sh_H(w)\vert-\e_H(\sh_H(w))\geq \sum_{\substack{s\in \sh_H(w)\\ \deg_H(s)\leq k-1}}(k-\deg_H(s))-1,$$
which proves the lemma.
\end{proof}

\begin{coro}\label{shadow-prop4} Let $w\in V_{\leq k-1}(H)$ with $\e_H(\sh_H(w))= (k-1)\vert \sh_H(w)\vert$. Then $w$ is the only vertex $v\in \sh_H(w)$ with $\deg_H(v)\leq k-1$ and furthermore $\deg_H(w)= k-1$.
\end{coro}
\begin{proof}The right-hand-side of the inequality in Lemma \ref{shadow-prop3} is 1, while all summands on the left-hand-side are positive and one of them is the summand $k-\deg_H(w)\geq 1$. This immediately implies the statement of this corollary.
\end{proof}

\begin{coro}\label{shadow-prop5} Let $w\in V_{\leq k-1}(H)$ with $\e_H(\sh_H(w))<(k-1)\vert \sh_H(w)\vert$. Then
$$\sum_{\substack{s\in \sh_H(w)\\ \deg_H(s)\leq k-1}}(k-\deg_H(s))\leq 2((k-1)\vert \sh_H(w)\vert-\e_H(\sh_H(w))).$$
\end{coro}
\begin{proof}We have $(k-1)\vert \sh_H(w)\vert-\e_H(\sh_H(w))\geq 1$ and hence by Lemma \ref{shadow-prop3}
$$\sum_{\substack{s\in \sh_H(w)\\ \deg_H(s)\leq k-1}}(k-\deg_H(s))\leq (k-1)\vert \sh_H(w)\vert-\e_H(\sh_H(w))+1\leq 2((k-1)\vert \sh_H(w)\vert-\e_H(\sh_H(w))).$$
\end{proof}

\begin{claim}\label{shadow-prop8} Let $w\in V_{\leq k-1}(H)$ with $\sh_H(w)\neq V(H)$. Then
$$V_{\leq k-1}(H-\sh_H(w))\su V_{\leq k-1}(H)\mi \lbrace w\rbrace.$$
\end{claim}
\begin{proof}From $\sh_H(w)\neq V(H)$ it is clear that $H-\sh_H(w)$ is a non-empty graph. Consider any vertex $v\in V_{\leq k-1}(H-\sh_H(w))$. If $v$ is adjacent to a vertex in $\sh_H(w)$, we have $\deg_{H-\sh_H(w)}(v)\geq k$ by property (III) in Definition \ref{defi-shadow}, a contradiction. Hence $v$ is not adjacent to a vertex in $\sh_H(w)$ and we have $\deg_{H}(v)=\deg_{H-\sh_H(w)}(v)\leq k-1$. Thus, $v\in V_{\leq k-1}(H)$. Furthermore $v\neq w$ since $v\not\in \sh_H(w)$.
\end{proof}

\begin{coro}\label{shadow-prop7} If $V_{\leq k-1}(H)=\lbrace w\rbrace$ and $\sh_H(w)\neq V(H)$, then $H-\sh_H(w)$ is a (non-empty) graph with minimum degree at least $k$.
\end{coro}
\begin{proof}This follows immediately from Claim \ref{shadow-prop8}.\end{proof}

\begin{lem}\label{shadow-prop6} Let $\tH$ be a graph containing $H$ as a proper induced subgraph, such that no vertex in $V(\tH)\mi V(H)$ is adjacent to any $D\in \C_H$. Let $w\in V_{\leq k-1}(\tH)\cap V(H)$ and suppose that $\e_H(\sh_H(w))= (k-1)\vert \sh_H(w)\vert$. Then $\sh_{\tH}(w)\su \sh_H(w)$ and no vertex in $V(\tH)\mi V(H)$ is adjacent to any vertex in $\sh_{\tH}(w)$.
\end{lem}

Let us clarify that the shadow $\sh_{\tH}(w)$ of $w$ in $\tH$ is defined with respect to the same collection $\C_H$ of subsets of $V(H)$ that we use for $H$ (with all the properties described in the beginning of Section \ref{sect4}).

\begin{proof} Note that $w\in V_{\leq k-1}(\tH)\cap V(H)$ already implies $w\in V_{\leq k-1}(H)$, since $H$ is an induced subgraph of $\tH$. Thus, $\sh_H(w)$ is well defined. By Corollary \ref{shadow-prop4} we know $\deg_H(w)= k-1$. Hence we must have $\deg_{\tH}(w)=\deg_H(w)= k-1$ and in particular $w$ is not adjacent to any vertex in $V(\tH)\mi V(H)$.

Let us consider Procedure \ref{proc-shadow} for determining the shadow $\sh_{\tH}(w)$ of $w$ in $\tH$. If we have multiple options, let us fix some specific choices, so that we obtain some fixed procedure starting with $Y=\lbrace w\rbrace$ and arriving at $Y=\sh_{\tH}(w)$.

We claim that during this procedure for determining the shadow $\sh_{\tH}(w)$ of $w$ in $\tH$ all the sets $Y$ have the following two properties:
\begin{itemize}
\item[($\alpha$)] No vertex in $Y$ is adjacent to a vertex in $V(\tH)\mi V(H)$.
\item[($\beta$)] It is possible to arrange Procedure \ref{proc-shadow} for determining the shadow $\sh_H(w)$ of $w$ in $H$ in such a way, that the set $Y$ also occurs during this procedure for $H$.
\end{itemize}

Since $w$ is not adjacent to any vertex in $V(\tH)\mi V(H)$, the set $Y=\lbrace w\rbrace$ in the beginning satisfies ($\alpha$) and it clearly also satisfies ($\beta$).

Now let $Y$ be any set occurring in the procedure for determining the shadow $\sh_{\tH}(w)$ of $w$ in $\tH$ and assume that $Y$ fulfills ($\alpha$) and ($\beta$). We want to show that after the next step, following the rules in Procedure \ref{proc-shadow} for the graph $\tH$, the next set still has the properties ($\alpha$) and ($\beta$).

Note that by ($\beta$) we in particular have $Y\su \sh_H(w)\su V(H)$.

Suppose the next step is adding some set $D\in \C_H$ with $D\cap Y=\es$, such that $D$ is adjacent to some vertex in $Y$. This step is definitely an allowed step in the procedure for the graph $H$ as well, so ($\beta$) is satisfied for $Y\cup D$. Since $Y$ satisfies ($\alpha$) and $D$ is not adjacent to any vertex in $V(\tH)\mi V(H)$ (by the assumptions of the lemma), the set $Y\cup D$ also satisfies ($\alpha$). So in this case we are done.

So we can assume that the next step is adding a vertex $v\in V(\tH)\mi Y$ with $v\not\in D$ for all $D\in \C_H$, such that $v$ is adjacent to a vertex in $Y$ and $\deg_{\tH-Y}(v)\leq k-1$. Since $Y$ is according to ($\alpha$) not adjacent to any vertex in $V(\tH)\mi V(H)$, we must have $v\in V(H)$. So $v\in V(H)\mi Y$ and $v$ is adjacent to a vertex in $Y$. Furthermore, $\deg_{H-Y}(v)\leq \deg_{\tH-Y}(v)\leq k-1$ (since $H-Y$ is an induced subgraph of $\tH-Y$). Thus, adding $v$ to $Y$ is also an allowed step in the procedure for determining the shadow $\sh_H(w)$ of $w$ in $H$. In particular, ($\beta$) is satisfied for $Y\cup \lbrace v\rbrace$ and it remains to show ($\alpha$) for $Y\cup \lbrace v\rbrace$.

Suppose that $v$ is adjacent to a vertex in $V(\tH)\mi V(H)$. Since $Y\su V(H)$, this vertex also lies in $V(\tH-Y)\mi V(H-Y)$. Hence we have $\deg_{H-Y}(v)\leq \deg_{\tH-Y}(v)-1\leq k-2$. But then
$$\e_H(Y\cup \lbrace v\rbrace)=\e_H(Y)+\e_{H-Y}(\lbrace v\rbrace)=\e_H(Y)+\deg_{H-Y}(v)\leq \e_H(Y)+(k-2),$$
and therefore
$$(k-1)\vert Y\cup \lbrace v\rbrace\vert-\e_H(Y\cup \lbrace v\rbrace)\geq (k-1)(\vert Y\vert+1)-(\e_H(Y)+(k-2))=(k-1)\vert Y\vert-\e_H(Y)+1,$$
so the quantity $(k-1)\vert Y\vert-\e_H(Y)$ increases by at least 1 when going from $Y$ to $Y\cup \lbrace v\rbrace$ in the procedure for determining the shadow $\sh_H(w)$ of $w$ in $H$. But this contradicts Corollary \ref{shadow-prop2}. Hence $v$ cannot be adjacent to a vertex in $V(\tH)\mi V(H)$. Since $Y$ satisfies ($\alpha$), this implies that $Y\cup \lbrace v\rbrace$ satisfies ($\alpha$) as well.

This finishes the proof that all sets $Y$ occurring in the procedure for determining the shadow $\sh_{\tH}(w)$ of $w$ in $\tH$ satisfy the two properties ($\alpha$) and ($\beta$).

In particular, the final set $Y=\sh_{\tH}(w)$ satisfies ($\alpha$) and ($\beta$). By ($\beta$) the set $\sh_{\tH}(w)$ is a possible set during the procedure of determining the shadow $\sh_{H}(w)$ of $w$ in $H$. Hence $\sh_{\tH}(w)\su \sh_{H}(w)$. By ($\alpha$) no vertex in $\sh_{\tH}(w)$ is adjacent to a vertex in $V(\tH)\mi V(H)$.\end{proof}

\section{Proof of Lemma \ref{mainlem}} \label{sect5}

The goal of this section is to finally prove Lemma \ref{mainlem}. The proof proceeds by induction on $v(H)$. We use the results from Section \ref{sect4}, but otherwise the inductive proof of Lemma \ref{mainlem} is mostly a long case-checking.

First, if $v(H)=1$, then $H$ just consists of a single vertex $w$ and no edges. Note that indeed $H$ does not contain a subgraph of minimum degree at least $k$. The collection $\C_H$ must be empty, because $\deg_H(w)<k$, so $w$ cannot be part of any member of $\C_H$. We can now take $S=\lbrace w\rbrace$, then $S\su V_{\leq k-1}(H)$ and $V_{\leq k-2}(H)\su S$ and
$$\sum_{s\in S}(k-\deg_H(s))=k\leq 2(k-1)= 2((k-1)v(H)-e(H))$$
(recall that $k\geq 2$). Since $V_{\leq k-1}(H)\mi S=\es$, we do not need to specify any sets $B_v$. Now if $\tH$ is a graph containing $H$ as a proper induced subgraph such that $V_{\leq k-1}(\tH)\su V_{\leq k-1}(H)\mi S=\es$, then $\tH$ itself has minimum degree at least $k$ and we can take $\tH'=\tH$. It is easy to see that in this case $\tH'=\tH$ satisfies the six properties in Lemma \ref{mainlem}.

Now let $H$ be a graph on $v(H)\geq 2$ vertices that does not have a subgraph of minimum degree at least $k$. Furthermore let $\C_H$ be a collection of disjoint non-empty subsets of $V(H)$, such that for each $D\in\C_H$ we have $\e_H(D)\leq (k-1)\vert D\vert+1$ and $\deg_H(v)\geq k$ for each $v\in D$.

By induction we can assume that Lemma \ref{mainlem} holds for all graphs on less than $v(H)$ vertices, and we have to prove Lemma \ref{mainlem} for $H$.

Since $H$ does not have a subgraph of minimum degree at least $k$, we know in particular that the minimum degree of $H$ itself is less than $k$. So we can fix a vertex $w\in V(H)$ with $\deg_H(w)\leq k-1$.

Consider the shadow $\sh_H(w)$ of $w$ in $H$. We distinguish two cases, namely whether $\sh_H(w)=V(H)$ (Case A) or whether $\sh_H(w)\subsetneq V(H)$ (Case B).

\textbf{Case A: $\sh_H(w)=V(H)$.}

By Corollary \ref{shadow-prop1} we have $\e_H(\sh_H(w))\leq (k-1)\vert \sh_H(w)\vert$. We distinguish two sub-cases, namely whether $\e_H(\sh_H(w))< (k-1)\vert \sh_H(w)\vert$ (Case A.1) or whether $\e_H(\sh_H(w))= (k-1)\vert \sh_H(w)\vert$ (Case A.2).

\textbf{Case A.1: $\sh_H(w)=V(H)$ and $\e_H(\sh_H(w))< (k-1)\vert \sh_H(w)\vert$.}

In this case we can take $S=V_{\leq k-1}(H)$, then clearly $S\su V_{\leq k-1}(H)$ and $ V_{\leq k-2}(H)\su S$ and by Corollary \ref{shadow-prop5}
\begin{multline*}
\sum_{s\in S}(k-\deg_H(s))=\sum_{\substack{s\in \sh_H(w)\\ \deg_H(s)\leq k-1}}(k-\deg_H(s))\leq 2((k-1)\vert \sh_H(w)\vert-\e_H(\sh_H(w)))\\
=2((k-1)v(H)-e(H)).
\end{multline*}
Since $V_{\leq k-1}(H)\mi S=\es$, we do not need to specify any sets $B_v$.
Now if $\tH$ is a graph containing $H$ as a proper induced subgraph such that $V_{\leq k-1}(\tH)\su V_{\leq k-1}(H)\mi S=\es$, then $\tH$ itself has minimum degree at least $k$ and we can take $\tH'=\tH$. It is easy to see that in this case $\tH'=\tH$ satisfies the six properties in Lemma \ref{mainlem}.

\textbf{Case A.2: $\sh_H(w)=V(H)$ and $\e_H(\sh_H(w))=(k-1)\vert \sh_H(w)\vert$.}

By Corollary \ref{shadow-prop4} the vertex $w$ is the only vertex $v\in \sh_H(w)=V(H)$ with $\deg_H(v)\leq k-1$ and furthermore $\deg_H(w)= k-1$. Thus, $V_{\leq k-1}(H)=\lbrace w\rbrace$ and $V_{\leq k-2}(H)=\es$.
Let us take $S=\es$, then clearly $S\su V_{\leq k-1}(H)$ with $ V_{\leq k-2}(H)\su S$ and
$$\sum_{s\in S}(k-\deg_H(s))=0=2((k-1)\vert \sh_H(w)\vert-\e_H(\sh_H(w)))= 2((k-1)v(H)-e(H)).$$
Now $V_{\leq k-1}(H)\mi S=\lbrace w\rbrace$ and let us take $B_w=\sh_H(w)=V(H)$.

Let $\tH$ be a graph containing $H$ as a proper induced subgraph with $V_{\leq k-1}(\tH) \su {V_{\leq k-1}(H)\mi S} = \lbrace w\rbrace$ and such that no vertex in $V(\tH)\mi V(H)$ is adjacent to any member of $\C_H$. We have to show the existence of an induced subgraph $\tH'$ of $\tH$ satisfying the properties (a) to (f) in Lemma \ref{mainlem}.

If $\tH$ itself has minimum degree at least $k$, then we can take $\tH'=\tH$ and all the properties are satisfied.

So let us assume that $\tH$ has minimum degree smaller than $k$. By $V_{\leq k-1}(\tH)\su V_{\leq k-1}(H)\mi S=\lbrace w\rbrace$ this implies $V_{\leq k-1}(\tH)=\lbrace w\rbrace$. We can apply Lemma \ref{shadow-prop6}, so $\sh_{\tH}(w)\su \sh_H(w)$ and no vertex in $V(\tH)\mi V(H)$ is adjacent to any vertex in $\sh_{\tH}(w)$. Recall that here the shadow $\sh_{\tH}(w)$ of $w$ in $\tH$ is defined with respect to the same collection $\C_H$ as for $H$.

Let us take $\tH'=\tH-\sh_{\tH}(w)$. Note that $\sh_{\tH}(w)\su \sh_H(w)=V(H)$ and therefore $\sh_{\tH}(w)\neq V(\tH)$. By Corollary \ref{shadow-prop7} applied to $\tH$ and $w$, the graph $\tH'=\tH-\sh_{\tH}(w)$ is a non-empty induced subgraph of $\tH$ and has minimum degree at least $k$. This already establishes property (a). Let us check the other properties:
\begin{itemize}
\item[(b)] By Corollary \ref{shadow-prop1} applied to $\tH$ and $w$ we have $\e_{\tH}(\sh_{\tH}(w))\leq (k-1)\vert \sh_{\tH}(w)\vert$. Hence
$$(k-1)v(\tH')-e(\tH')=(k-1)(v(\tH)-\vert \sh_{\tH}(w)\vert)-(e(\tH)-\e_{\tH}(\sh_{\tH}(w)))\leq (k-1)v(\tH)-e(\tH).$$
\item[(c)] $V(\tH)\mi V(\tH')=\sh_{\tH}(w)\su \sh_H(w)=B_w=\bigcup_{v\in V_{\leq k-1}(\tH)}B_v$.
\item[(d)] As established above from Lemma \ref{shadow-prop6}, no vertex in $V(\tH)\mi V(\tH')=\sh_{\tH}(w)$ is adjacent to any vertex in $V(\tH)\mi V(H)$.
\item[(e)] For each $D\in \C_H$, by property (II) in Definition \ref{defi-shadow}  we either have $D\su \sh_{\tH}(w)$ or we have $D\cap \sh_{\tH}(w)=\es$. Thus, either $D\cap V(\tH')=\es$ or $D\su V(\tH')$.
\item[(f)] If $D\in \C_H$ and $D\su V(\tH')$, then $D\cap \sh_{\tH}(w)=\es$. So by property (IV) in Definition \ref{defi-shadow} the set $D$ is not adjacent to any vertex in $\sh_{\tH}(w)=V(\tH)\mi V(\tH')$.
\end{itemize}

\textbf{Case B: $\sh_H(w)\subsetneq V(H)$.}

In this case $F=H-\sh_H(w)$ is a non-empty subgraph of $H$ and $F$ has fewer vertices than $H$. Note that $F$ does not have a subgraph of minimum degree at least $k$, since according to the assumptions of Lemma \ref{mainlem} the graph $H$ has no such subgraph.

Let $\C_F$ be the collection of those $D\in \C_H$ with $D\su V(F)$, i.e.
\begin{equation}\label{F-collection}
\C_F=\lbrace D\in \C_H\,\vert\,D\cap \sh_H(w)=\es\rbrace=\lbrace D\in \C_H\,\vert\,D\not\su\sh_H(w)\rbrace,
\end{equation}
using that $\sh_H(w)$ has property (II) from Definition \ref{defi-shadow}.

Clearly $\C_F$ is a collection of disjoint non-empty subsets of $V(F)$ and for each $D\in \C_F$ we have
$$\e_F(D)\leq \e_H(D)\leq (k-1)\vert D\vert+1.$$

Furthermore for each $D\in \C_F\su \C_H$ we have $D\cap \sh_H(w)=\es$ and by property (IV) from Definition \ref{defi-shadow} this implies that $D$ is not adjacent to any vertex in $\sh_H(w)$. Hence for every $v\in D$ we have
$$\deg_F(v)=\deg_{H-\sh_H(w)}(v)=\deg_{H}(v)\geq k.$$

Thus, the graph $F$ together with the collection $\C_F$ of subsets of $V(F)$ satisfies the assumptions of Lemma \ref{mainlem}. By the induction assumption, we can  apply Lemma \ref{mainlem} to $F$ (with $\C_F$). We find a subset $S_F\su V_{\leq k-1}(F)$ with $ V_{\leq k-2}(F)\su S_F$ and
\begin{equation}\label{F-SF-ineq}
\sum_{s\in S_F}(k-\deg_F(s))\leq 2((k-1)v(F)-e(F))
\end{equation}
as well as disjoint subsets $B_v\su V(F)$ for each vertex $v\in V_{\leq k-1}(F)\mi S_F=V_{k-1}(F)\mi S_F$, such that the conclusion of Lemma \ref{mainlem} holds.

Note that
\begin{multline}\label{F-edge-defect}
(k-1)v(F)-e(F)=(k-1)(v(H)-\vert \sh_{H}(w)\vert)-(e(H)-\e_{H}(\sh_{H}(w)))\\
=(k-1)v(H)-e(H)-((k-1)\vert \sh_{H}(w)\vert-\e_{H}(\sh_{H}(w))).
\end{multline}

As $F$ is an induced subgraph of $H$, for all $v\in V(F)$ we have
\begin{equation}\label{F-deg-H}
\deg_F(v)\leq \deg_H(v).
\end{equation}

\begin{claim}\label{F-k-1-set} $V_{\leq k-1}(H)=V_{\leq k-1}(F)\cup \lbrace v\in \sh_H(w)\,\vert\, \deg_H(v)\leq k-1\rbrace$.
\end{claim}
\begin{proof}First, let $v$ be an element of the left-hand-side, i.e.\ $v\in V(H)$ and $\deg_H(v)\leq k-1$. If $v\in \sh_H(w)$, then $v$ is clearly contained in the right-hand side. Otherwise $v\in V(F)$ and $\deg_F(v)\leq \deg_H(v)\leq k-1$ by (\ref{F-deg-H}), hence $v\in V_{\leq k-1}(F)$.

For the other inclusion note that obviously
$$\lbrace v\in \sh_H(w)\,\vert\, \deg_H(v)\leq k-1\rbrace\su V_{\leq k-1}(H)$$
and that
$$V_{\leq k-1}(F)=V_{\leq k-1}(H-\sh_H(w))\su V_{\leq k-1}(H)$$
by Claim \ref{shadow-prop8}.\end{proof}

\begin{claim}\label{F-C-set} No vertex in $V(H)\mi V(F)$ is adjacent to any member of $\C_F$.
\end{claim}
\begin{proof}Let $D\in \C_F$, then by (\ref{F-collection}) we have $D\cap\sh_H(w)=\es$. Since $\sh_H(w)$ has property (IV) in Definition \ref{defi-shadow}, this implies that $D$ is not adjacent to any vertex in $\sh_H(w)=V(H)\mi V(F)$.
\end{proof}

By Corollary \ref{shadow-prop1} we again have $\e_H(\sh_H(w))\leq (k-1)\vert \sh_H(w)\vert$. As before, let us distinguish two sub-cases, namely whether $\e_H(\sh_H(w))< (k-1)\vert \sh_H(w)\vert$ (Case B.1) or whether $\e_H(\sh_H(w))= (k-1)\vert \sh_H(w)\vert$ (Case B.2).

\textbf{Case B.1: $\sh_H(w)\subsetneq V(H)$ and $\e_H(\sh_H(w))< (k-1)\vert \sh_H(w)\vert$.}

In this case let us take
$$S_H=S_F\cup \lbrace v\in \sh_H(w)\,\vert\, \deg_H(v)\leq k-1\rbrace.$$
Since $S_F\su V_{\leq k-1}(F)\su V_{\leq k-1}(H)$ (see Claim \ref{F-k-1-set}), it is clear that $S_H\su V_{\leq k-1}(H)$.

Let us check $V_{\leq k-2}(H)\su S_H$: Let $v\in V(H)$ with $\deg_H(v)\leq k-2$. If $v\in \sh_H(w)$, then we clearly have $v\in S_H$. Otherwise $v\in V(F)$ and by (\ref{F-deg-H}) we have $\deg_F(v)\leq \deg_H(v)\leq k-2$, so $v\in V_{\leq k-2}(F)\su S_F\su S_H$. Thus, indeed $V_{\leq k-2}(H)\su S_H$.

Furthermore note that
$$\sum_{s\in S_H}(k-\deg_H(s))=\sum_{s\in S_F}(k-\deg_H(s))+\sum_{\substack{s\in \sh_H(w)\\ \deg_H(s)\leq k-1}}(k-\deg_H(s)).$$
By (\ref{F-deg-H}) and Corollary \ref{shadow-prop5} (note that the assumption of this corollary is fulfilled by the assumption of Case B.1) we obtain
$$\sum_{s\in S_H}(k-\deg_H(s))\leq\sum_{s\in S_F}(k-\deg_F(s))+2((k-1)\vert \sh_H(w)\vert-\e_H(\sh_H(w))).$$
Plugging in (\ref{F-SF-ineq}) yields
$$\sum_{s\in S_H}(k-\deg_H(s))\leq 2((k-1)v(F)-e(F))+2((k-1)\vert \sh_H(w)\vert-\e_H(\sh_H(w))),$$
and, together with (\ref{F-edge-defect}), this gives
$$\sum_{s\in S_H}(k-\deg_H(s))\leq 2((k-1)v(H)-e(H)).$$

All in all, $S_H\su V_{\leq k-1}(H)$ has all the desired properties to act as the set $S$ for $H$.

Furthermore note that by Claim \ref{F-k-1-set} we have
\begin{equation}\label{F-S-difference-1}
V_{\leq k-1}(H)\mi S_H=V_{\leq k-1}(F)\mi S_F.
\end{equation}
Thus, for each $v\in V_{\leq k-1}(H)\mi S_H=V_{\leq k-1}(F)\mi S_F$ we already have a set $B_v\su V(F)\su V(H)$ coming from the application of Lemma \ref{mainlem} to $F$. We can just keep those sets $B_v\su V(H)$ for each $v\in V_{\leq k-1}(H)\mi S_H=V_{\leq k-1}(F)\mi S_F$. They are still disjoint.

Now let $\tH$ be a graph containing $H$ as a proper induced subgraph with $V_{\leq k-1}(\tH)\su V_{\leq k-1}(H)\mi S_H$ such that no vertex in $V(\tH)\mi V(H)$ is adjacent to any member of $\C_H$. We need to find a (non-empty) induced subgraph $\tH'$ of $\tH$ with the properties (a) to (f) listed in Lemma \ref{mainlem}.

Note that $\tH$ also contains $F$ as a proper induced subgraph. By (\ref{F-S-difference-1}) we have 
$$V_{\leq k-1}(\tH)\su V_{\leq k-1}(H)\mi S_H=V_{\leq k-1}(F)\mi S_F.$$
Furthermore no vertex in $V(\tH)\mi V(H)$ is adjacent to any member of $\C_F\su \C_H$ and by Claim \ref{F-C-set} no vertex in $V(H)\mi V(F)$ is adjacent to any member of $\C_F$. Thus, no vertex in $V(\tH)\mi V(F)$ is adjacent to any member of $\C_F$.

So the graph $\tH$ satisfies all conditions in Lemma \ref{mainlem} for $F$ (together with the collection $\C_F$). So by the conclusion of Lemma \ref{mainlem} for $F$ we can find a (non-empty) induced subgraph $\tH'$ of $\tH$ with the following six properties (these are the properties (a) to (f) but with respect to $F$ instead of $H$):
\begin{itemize}
\item[(a$_F$)] The minimum degree of $\tH'$ is at least $k$.
\item[(b$_F$)] $(k-1)v(\tH')-e(\tH')\leq (k-1)v(\tH)-e(\tH)$.
\item[(c$_F$)] $V(\tH)\mi V(\tH')\su \bigcup_{v\in V_{\leq k-1}(\tH)}B_v$, so in particular $V(\tH)\mi V(\tH')\su V(F)$.
\item[(d$_F$)] No vertex in $V(\tH)\mi V(\tH')$ is adjacent to any vertex in $V(\tH)\mi V(F)$.
\item[(e$_F$)] For each $D\in \C_F$ either $D\su V(\tH')$ or $D\cap V(\tH')=\es$.
\item[(f$_F$)] If $D\in \C_F$ and $D\su V(\tH')$, then $D$ is not adjacent to any vertex in $V(\tH)\mi V(\tH')$.
\end{itemize}
We have to show that there is an induced subgraph of $\tH$ fulfilling the six properties (a) to (f) in Lemma \ref{mainlem} (with respect to $H$ and the collection $\C_H$). Let us just take the (non-empty) induced subgraph $\tH'$ of $\tH$ with the properties (a$_F$) to (f$_F$) above. Let us check that this same graph $\tH'$ also satisfies (a) to (f). For (a) and (b) this is clear, because they are identical with (a$_F$) and (b$_F$). Property (c) is identical with property (c$_F$) as well (recall that the sets $B_v$ are by definition the same for $F$ and for $H$ and $B_v\su V(F)\su V(H)$). Property (d) follows directly from property (d$_F$),  because $V(\tH)\mi V(H)\su V(\tH)\mi V(F)$. It remains to check properties (e) and (f):
\begin{itemize}
\item[(e)] Let $D\in \C_H$ and we have to show that $D\su V(\tH')$ or $D\cap V(\tH')=\es$. If $D\in \C_F$, this is clear from property (e$_F$). Otherwise, by (\ref{F-collection}) we have $D\su \sh_H(w)=V(H)\mi V(F)\su V(\tH)\mi V(F)$. Since $V(\tH)\mi V(\tH')\su V(F)$ by (c$_F$), this implies $D\su V(\tH')$.
\item[(f)] Let $D\in \C_H$ and $D\su V(\tH')$. We have to show that $D$ is not adjacent to any vertex in $V(\tH)\mi V(\tH')$. If $D\in \C_F$, this is clear from property (f$_F$). Otherwise,  by (\ref{F-collection}) we have $D\su \sh_H(w)=V(H)\mi V(F)\su V(\tH)\mi V(F)$. So by (d$_F$) the set $D$ is not adjacent to any vertex in $V(\tH)\mi V(\tH')$.
\end{itemize}

\textbf{Case B.2: $\sh_H(w)\subsetneq V(H)$ and $\e_H(\sh_H(w))= (k-1)\vert \sh_H(w)\vert$.}

By Corollary \ref{shadow-prop4} the vertex $w$ is the only vertex $v\in \sh_H(w)$ with $\deg_H(v)\leq k-1$ and furthermore $\deg_H(w)= k-1$. In particular, Claim \ref{F-k-1-set} implies
\begin{equation}\label{F-k-1-set2}
V_{\leq k-1}(H)=V_{\leq k-1}(F)\cup \lbrace w\rbrace.
\end{equation}

Let us take $S_H=S_F$.

Then $S_H=S_F\su V_{\leq k-1}(F)\su V_{\leq k-1}(H)$ by (\ref{F-k-1-set2}).

Let us check $V_{\leq k-2}(H)\su S_H$: Let $v\in V(H)$ with $\deg_H(v)\leq k-2$. Then clearly $v\in V_{\leq k-1}(H)$ and furthermore $v\neq w$ since $\deg_H(w)= k-1$. Thus, (\ref{F-k-1-set2}) implies $v\in V_{\leq k-1}(F)$ and by (\ref{F-deg-H}) we have $\deg_F(v)\leq \deg_H(v)\leq k-2$. Thus, $v\in V_{\leq k-2}(F)\su S_F=S_H$.

By (\ref{F-edge-defect}) the assumption $\e_H(\sh_H(w))= (k-1)\vert \sh_H(w)\vert$ of Case B.2 implies
$$(k-1)v(F)-e(F)=(k-1)v(H)-e(H).$$
Now (\ref{F-SF-ineq}) reads
$$\sum_{s\in S_F}(k-\deg_F(s))\leq 2((k-1)v(H)-e(H))$$
and by (\ref{F-deg-H}) it implies
$$\sum_{s\in S_H}(k-\deg_H(s))=\sum_{s\in S_F}(k-\deg_H(s))\leq \sum_{s\in S_F}(k-\deg_F(s))\leq 2((k-1)v(H)-e(H)).$$

All in all, $S_H\su V_{\leq k-1}(H)$ has all the desired properties to act as the set $S$ for $H$.

Furthermore note that by (\ref{F-k-1-set2}) we have
\begin{equation}\label{F-S-difference-2}
V_{\leq k-1}(H)\mi S_H=(V_{\leq k-1}(F)\mi S_F)\cup \lbrace w\rbrace.
\end{equation}
For each $v\in V_{\leq k-1}(F)\mi S_F$ we already have a set $B_v\su V(F)\su V(H)$ coming from the application of Lemma \ref{mainlem} to $F$. We can just keep those sets $B_v\su V(H)$ for each $v\in V_{\leq k-1}(F)\mi S_F$. They are still disjoint. For $v=w$ let us set $B_w=\sh_H(w)\su V(H)$. Since $B_v\su V(F)=V(H)\mi \sh_H(w)$ for $v\in V_{\leq k-1}(F)\mi S_F$, the set $B_w$ is disjoint from all the other $B_v$. So this defines disjoint sets $B_v\in V(H)$ for each $v\in V_{\leq k-1}(H)\mi S_H$.

Now let $\tH$ be a graph containing $H$ as a proper induced subgraph with $V_{\leq k-1}(\tH)\su V_{\leq k-1}(H)\mi S_H$ such that no vertex in $V(\tH)\mi V(H)$ is adjacent to any member of $\C_H$. We need to find a (non-empty) induced subgraph $\tH'$ of $\tH$ with the properties (a) to (f) listed in Lemma \ref{mainlem}.

In order to find a suitable $\tH'$, we distinguish two cases again: $w\not\in V_{\leq k-1}(\tH)$ (Case B.2.a) or $w\in V_{\leq k-1}(\tH)$ (Case B.2.b).

\textbf{Case B.2.a: $\sh_H(w)\subsetneq V(H)$, $\e_H(\sh_H(w))= (k-1)\vert \sh_H(w)\vert$ and $w\not\in V_{\leq k-1}(\tH)$.}

Since $\tH$ contains $H$ as a proper induced subgraph, it also contains $F$ as a proper induced subgraph. By (\ref{F-S-difference-2}) we have 
$$V_{\leq k-1}(\tH)\su V_{\leq k-1}(H)\mi S_H=(V_{\leq k-1}(F)\mi S_F)\cup \lbrace w\rbrace$$
and by the assumption $w\not\in V_{\leq k-1}(\tH)$ of Case B.2.a this implies $V_{\leq k-1}(\tH)\su V_{\leq k-1}(F)\mi S_F$.

Furthermore no vertex in $V(\tH)\mi V(H)$ is adjacent to any member of $\C_F\su \C_H$ and by Claim \ref{F-C-set} no vertex in $V(H)\mi V(F)$ is adjacent to any member of $\C_F$. Thus, no vertex in $V(\tH)\mi V(F)$ is adjacent to any member of $\C_F$.

So the graph $\tH$ satisfies all conditions in Lemma \ref{mainlem} for $F$ (together with the collection $\C_F$). So by the conclusion of Lemma \ref{mainlem} for $F$ we can find a (non-empty) induced subgraph $\tH'$ of $\tH$ with the following six properties (these are again the properties (a) to (f) but with respect to $F$ instead of $H$):
\begin{itemize}
\item[(a$_F$)] The minimum degree of $\tH'$ is at least $k$.
\item[(b$_F$)] $(k-1)v(\tH')-e(\tH')\leq (k-1)v(\tH)-e(\tH)$.
\item[(c$_F$)] $V(\tH)\mi V(\tH')\su \bigcup_{v\in V_{\leq k-1}(\tH)}B_v$, so in particular $V(\tH)\mi V(\tH')\su V(F)$.
\item[(d$_F$)] No vertex in $V(\tH)\mi V(\tH')$ is adjacent to any vertex in $V(\tH)\mi V(F)$.
\item[(e$_F$)] For each $D\in \C_F$ either $D\su V(\tH')$ or $D\cap V(\tH')=\es$.
\item[(f$_F$)] If $D\in \C_F$ and $D\su V(\tH')$, then $D$ is not adjacent to any vertex in $V(\tH)\mi V(\tH')$.
\end{itemize}
We have to show that there is an induced subgraph of $\tH$ fulfilling the six properties (a) to (f) in Lemma \ref{mainlem} (with respect to $H$ and the collection $\C_H$). Let us just take the (non-empty) induced subgraph $\tH'$ of $\tH$ with the properties (a$_F$) to (f$_F$) above. Let us check that this same graph $\tH'$ also satisfies (a) to (f). For (a) and (b) this is clear, because they are identical with (a$_F$) and (b$_F$). Property (c) is identical with property (c$_F$) as well (recall that the sets $B_v$ for $v\neq w$ are by definition the same for $F$ and for $H$). Property (d) follows directly from property (d$_F$),  because $V(\tH)\mi V(H)\su V(\tH)\mi V(F)$. It remains to check properties (e) and (f):
\begin{itemize}
\item[(e)] Let $D\in \C_H$ and we have to show that $D\su V(\tH')$ or $D\cap V(\tH')=\es$. If $D\in \C_F$, this is clear from property (e$_F$). Otherwise, by (\ref{F-collection}) we have $D\su \sh_H(w)=V(H)\mi V(F)\su V(\tH)\mi V(F)$. Since $V(\tH)\mi V(\tH')\su V(F)$ by (c$_F$), this implies $D\su V(\tH')$.
\item[(f)] Let $D\in \C_H$ and $D\su V(\tH')$. We have to show that $D$ is not adjacent to any vertex in $V(\tH)\mi V(\tH')$. If $D\in \C_F$, this is clear from property (f$_F$). Otherwise,  by (\ref{F-collection}) we have $D\su \sh_H(w)=V(H)\mi V(F)\su V(\tH)\mi V(F)$. So by (d$_F$) the set $D$ is not adjacent to any vertex in $V(\tH)\mi V(\tH')$.
\end{itemize}

\textbf{Case B.2.b: $\sh_H(w)\subsetneq V(H)$, $\e_H(\sh_H(w))= (k-1)\vert \sh_H(w)\vert$ and $w\in V_{\leq k-1}(\tH)$.}

Recall that no vertex in $V(\tH)\mi V(H)$ is adjacent to any member of $\C_H$. So we can apply Lemma \ref{shadow-prop6} and hence $\sh_{\tH}(w)\su \sh_H(w)$ and no vertex in $V(\tH)\mi V(H)$ is adjacent to any vertex in $\sh_{\tH}(w)$. Here the shadow $\sh_{\tH}(w)$ of $w$ in $\tH$ is defined with respect to the collection $\C_H$.

Set
$$\tH_F=\tH-\sh_{\tH}(w).$$
Here, we cannot use the graph $\tH$ itself in Lemma \ref{mainlem} for the graph $F$ as in the previous two cases (Case B.1 and Case B.2.a). But let us check that we can use the graph $\tH_F$ instead.

As $\sh_{\tH}(w)\su \sh_H(w)$, the graph $\tH_F=\tH-\sh_{\tH}(w)$ does indeed contain the graph $F=H-\sh_{H}(w)$ as an induced subgraph. Since $\sh_{\tH}(w)\su \sh_H(w)\su V(H)$, we have
\begin{equation}\label{tHFohneF}
V(\tH)\mi V(H)\su V(\tH_F)\mi V(F)
\end{equation}
and since $H$ is a proper induced subgraph of $\tH$, we can conclude that $F$ is also a proper induced subgraph of $\tH_F$.

Recall that $\sh_{\tH}(w)\su \sh_H(w)\su V(H)\subsetneq V(\tH)$. Thus, Claim \ref{shadow-prop8} applied to the graph $\tH$ and the vertex $w$ implies that
\begin{equation}\label{tHF-degk-1}
V_{\leq k-1}(\tH_F)=V_{\leq k-1}(\tH-\sh_{\tH}(w))\su V_{\leq k-1}(\tH)\mi \lbrace w\rbrace.
\end{equation}
From $V_{\leq k-1}(\tH)\su V_{\leq k-1}(H)\mi S_H$ and (\ref{F-S-difference-2}) we have
$$V_{\leq k-1}(\tH)\su V_{\leq k-1}(H)\mi S_H= (V_{\leq k-1}(F)\mi S_F)\cup \lbrace w\rbrace.$$
Together with (\ref{tHF-degk-1}) this yields
$$V_{\leq k-1}(\tH_F)\su V_{\leq k-1}(F)\mi S_F.$$

Furthermore no vertex in $V(\tH)\mi V(H)$ is adjacent to any member of $\C_F\su \C_H$ and by Claim \ref{F-C-set} no vertex in $V(H)\mi V(F)$ is adjacent to any member of $\C_F$. Thus, no vertex in $V(\tH)\mi V(F)$ is adjacent to any member of $\C_F$. In particular, no vertex in $V(\tH_F)\mi V(F)$ is adjacent to any member of $\C_F$.

So the graph $\tH_F$ indeed satisfies all conditions in Lemma \ref{mainlem} for $F$ (together with the collection $\C_F$). So by the conclusion of Lemma \ref{mainlem} for $F$ we can find a (non-empty) induced subgraph $\tH'$ of $\tH_F$ with the following six properties:
\begin{itemize}
\item[(a$_F$)] The minimum degree of $\tH'$ is at least $k$.
\item[(b$_F$)] $(k-1)v(\tH')-e(\tH')\leq (k-1)v(\tH_F)-e(\tH_F)$.
\item[(c$_F$)] $V(\tH_F)\mi V(\tH')\su \bigcup_{v\in V_{\leq k-1}(\tH_F)}B_v$, so in particular $V(\tH_F)\mi V(\tH')\su V(F)$.
\item[(d$_F$)] No vertex in $V(\tH_F)\mi V(\tH')$ is adjacent to any vertex in $V(\tH_F)\mi V(F)$.
\item[(e$_F$)] For each $D\in \C_F$ either $D\su V(\tH')$ or $D\cap V(\tH')=\es$.
\item[(f$_F$)] If $D\in \C_F$ and $D\su V(\tH')$, then $D$ is not adjacent to any vertex in $V(\tH_F)\mi V(\tH')$.
\end{itemize}

We have to show that there is an induced subgraph of $\tH$ fulfilling the six properties (a) to (f) in Lemma \ref{mainlem} (with respect to $H$ and the collection $\C_H$). Note that $\tH'$ is a non-empty induced subgraph of $\tH_F$ and hence also of $\tH$. Let us now check that this graph $\tH'$ satisfies (a) to (f) and is therefore the desired induced subgraph of $\tH$.
\begin{itemize}
\item[(a)] This statement is identical with (a$_F$).
\item[(b)] Note that
\begin{multline*}
(k-1)v(\tH_F)-e(\tH_F)=(k-1)(v(\tH)-\vert \sh_{\tH}(w)\vert)-(e(\tH)-\e_{\tH}(\sh_{\tH}(w)))\\
=(k-1)v(\tH)-e(\tH)-((k-1)\vert \sh_{\tH}(w)\vert-\e_{\tH}(\sh_{\tH}(w))).
\end{multline*}
By Corollary \ref{shadow-prop1} applied to $\tH$ we have $\e_{\tH}(\sh_{\tH}(w))\leq (k-1)\vert \sh_{\tH}(w)\vert$ and therefore
$$(k-1)v(\tH_F)-e(\tH_F)\leq (k-1)v(\tH)-e(\tH).$$
Thus, by (b$_F$)
$$(k-1)v(\tH')-e(\tH')\leq (k-1)v(\tH_F)-e(\tH_F)\leq (k-1)v(\tH)-e(\tH).$$
\item[(c)] Since $V(\tH)=\sh_{\tH}(w)\cup V(\tH_F)$, we have using (c$_F$) and $B_w=\sh_{\tH}(w)$
$$V(\tH)\mi V(\tH')\su \sh_{\tH}(w)\cup (V(\tH_F)\mi V(\tH'))\su B_w\cup \bigcup_{v\in V_{\leq k-1}(\tH_F)}B_v=\bigcup_{v\in V_{\leq k-1}(\tH_F)\cup \lbrace w\rbrace}B_v.$$
By (\ref{tHF-degk-1}) and the assumption $w\in V_{\leq k-1}(\tH)$ of Case B.2.b we have $V_{\leq k-1}(\tH_F)\cup \lbrace w\rbrace\su V_{\leq k-1}(\tH)$ and hence
$$V(\tH)\mi V(\tH')\su \bigcup_{v\in V_{\leq k-1}(\tH)}B_v,$$
(and in particular $V(\tH)\mi V(\tH')\su V(H)$ as $B_v\su V(H)$ for all $v\in V_{\leq k-1}(\tH)$).
\item[(d)] Recall that no vertex in $\sh_{\tH}(w)$ is adjacent to any vertex in $V(\tH)\mi V(H)$. By (d$_F$) no vertex in $V(\tH_F)\mi V(\tH')$ is adjacent to any vertex in $V(\tH_F)\mi V(F)$. By (\ref{tHFohneF}) this implies that no vertex in $V(\tH_F)\mi V(\tH')$ is adjacent to any vertex in $V(\tH)\mi V(H)\su V(\tH_F)\mi V(F)$.
Hence no vertex in $V(\tH)\mi V(\tH')= \sh_{\tH}(w)\cup (V(\tH_F)\mi V(\tH'))$ is adjacent to any vertex in $V(\tH)\mi V(H)$.
\item[(e)] Let $D\in \C_H$ and we have to show that $D\su V(\tH')$ or $D\cap V(\tH')=\es$. If $D\in \C_F$, this is clear from property (e$_F$). Otherwise, by (\ref{F-collection}) we have $D\su \sh_H(w)=V(H)\mi V(F)\su V(\tH)\mi V(F)$.  By property (II) in Definition \ref{defi-shadow} for the shadow $\sh_{\tH}(w)$ of $w$ in $\tH$ we know $D\su \sh_{\tH}(w)$ or $D\cap \sh_{\tH}(w)=\es$. If $D\su \sh_{\tH}(w)=V(\tH)\mi V(\tH_F)$, then $D$ is disjoint from $V(\tH')\su V(\tH_F)$. If $D\cap \sh_{\tH}(w)=\es$, then $D\su V(\tH_F)$ and by $D\su V(\tH)\mi V(F)$ we obtain $D\su V(\tH_F)\mi V(F)$. On the other hand $V(\tH_F)\mi V(\tH')\su V(F)$ by (c$_F$) and hence $D\su V(\tH_F)\mi V(F)\su  V(\tH')$. So in any case we have shown $D\su V(\tH')$ or $D\cap V(\tH')=\es$.
\item[(f)] Let $D\in \C_H$ and $D\su V(\tH')$. We have to show that $D$ is not adjacent to any vertex in $V(\tH)\mi V(\tH')$.

Note that $D\su V(\tH')\su V(\tH_F)$ implies $D\cap \sh_{\tH}(w)=\es$ and by property (IV) in Definition \ref{defi-shadow} we can conclude that $D$ is not adjacent to any vertex in $\sh_{\tH}(w)=V(\tH)\mi V(\tH_F)$. So it remains to show that $D$ is not adjacent to any vertex in $V(\tH_F)\mi V(\tH')$.

If $D\in \C_F$, then by (f$_F$) the set $D$ is not adjacent to any vertex in $V(\tH_F)\mi V(\tH')$ and we are done.

So let us now assume $D\not\in \C_F$, then by (\ref{F-collection}) we have $D\su \sh_H(w)=V(H)\mi V(F)$. On the other hand $D\su V(\tH')\su V(\tH_F)$, hence $D\su V(\tH_F)\mi V(F)$. So by (d$_F$) we obtain that $D$ is not adjacent to any vertex in $V(\tH_F)\mi V(\tH')$.
\end{itemize}
This finishes the proof.

\textit{Acknowledgements.} The author would like to thank her advisor Jacob Fox for suggesting this problem and for very helpful comments on earlier versions of this paper. Furthermore, the author is grateful to the anonymous referees for their useful comments and suggestions.

\section*{Appendix}
Here we provide a proof for Lemma \ref{deg-k-vertices}. Our proof is similar to the proof of Lemma 4 in \cite{erdos2}.

Fix $n$ and assume there are counterexamples to Lemma \ref{deg-k-vertices}. We choose a counterexample $H$ with minimum edge number, i.e.\ $H$ is a graph on $n$ vertices with minimum degree at least $k$ and with at most $\frac{n}{3k}$ vertices of degree $k$ and we assume that Lemma \ref{deg-k-vertices} holds for any graph with fewer than $e(H)$ edges.

Suppose $H$ has an edge between two vertices of degree at least $k+2$. After deleting that edge, the resulting graph still has $n$ vertices, minimum degree at least $k$ and at most $\frac{n}{3k}$ vertices of degree $k$, but it has fewer than $e(H)$ edges. Therefore it must have a subgraph on at most $\left(1-\frac{1}{27k^{2}}\right)n$ vertices with minimum degree at least $k$. But this subgraph is also a subgraph of $H$, which gives a contradiction to $H$ being a counterexample to Lemma \ref{deg-k-vertices}.

Hence $H$ has no edge between any two vertices of degree at least $k+2$. On the other hand at most $(k+1)\vert V_{\leq k+1}(H)\vert$ edges can be incident with a vertex of degree at most $k+1$, so
$$e(H)\leq (k+1)\vert V_{\leq k+1}(H)\vert\leq (k+1)n.$$

Let $T_1\su V(H)$ be the set of those vertices that are adjacent to a vertex of degree $k$. Then
$$\vert T_1\vert\leq k\vert V_{k}(H)\vert\leq \frac{k}{3k}n= \frac{1}{3}n.$$

Let $T_2=V_{\geq 9k}\su V(H)$ be the set of vertices of degree at least $9k$. Note that the sum of the degrees of all vertices of $H$ is $2e(H)\leq 2(k+1)n\leq 3kn$ (recall that $k\geq 2$). Hence $\vert T_2\vert\leq \frac{1}{3}n$.

Finally let $T=V(H)\mi (T_1\cup T_2)$, then $\vert T\vert\geq n-\vert T_1\vert-\vert T_2\vert\geq \frac{1}{3}n$.

Now colour an arbitrary vertex in $T$ red. This vertex has at most $9k-1$ neighbours in $H$ and each of them has degree at least $k+1$. So for each neighbour $v$ of the red vertex we can colour $k$ neighbours of $v$ in blue. After doing that, choose an uncoloured vertex in $T$, colour it red and for each of its neighbours $v$ ensure (by possibly colouring more vertices blue) that $v$ has at least $k$ blue neighbours. Repeat this process as long as there are uncoloured vertices in $T$. Whenever colouring a vertex $w$ of $T$ in red, it is indeed possible to ensure that each of its neighbours $v$ is adjacent to at least $k$ blue vertices: If $v$ already has a red neighbour from a previous step, then this already ensures that $v$ has at least $k$ blue neighbours. If $w$ is the first red neighbour of $v$, then we have $\deg_H(v)\geq k+1$ since $w\not\in T_1$. Hence $v$ has at least $k$ more neighbours besides $v$ all of which are either blue or uncoloured. By colouring up to $k$ uncoloured neighbours of $v$ in blue, we can ensure that $v$ has at least $k$ blue neighbours.

Note that for every vertex  $w\in T$ which we colour red, $w$ has at most $9k-1$ neighbours $v$ and for each of them we colour at most $k$ vertices blue. So in each step we colour one vertex red and at most $9k^{2}-k\leq 9k^{2}-1$ vertices blue. All in all we colour at most $9k^{2}$ vertices in each step and one of them we colour red. Since we repeat the process until all of $T$ is coloured, we make at least
$\frac{\vert T\vert}{9k^{2}}\geq \frac{1}{27k^{2}}n$
steps and therefore at least $\frac{1}{27k^{2}}n$ vertices get coloured red.

Now consider the subgraph $H'$ of $H$ obtained by deleting all red vertices. This graph has at most $\left(1-\frac{1}{27k^{2}}\right)n$ vertices. Every vertex $v\in V(H')$ that is adjacent to a red vertex in $H$ is also adjacent to at least $k$ blue vertices and these are all part of $H'$. Hence $v$ has degree at least $k$ in $H'$. Every vertex $v\in V(H')$ that is not adjacent to a red vertex in $H$ satisfies $\deg_{H'}(v)=\deg_{H}(v)\geq k$. Thus, the graph $H'$ has minimum degree at least $k$. But this contradicts $H$ being a counterexample, which completes the proof.
\end{document}